\documentclass[draftcls, onecolumn]{IEEEtran}
\usepackage{afterpage,cite,graphicx,url,amssymb,amsthm,epstopdf,threeparttable,multirow,algorithm,algorithmic}
\usepackage[tight,footnotesize]{subfigure}
\usepackage[usenames]{color}
\usepackage[normalem]{ulem}
\usepackage[cmex10]{amsmath}
\IEEEoverridecommandlockouts

\newcommand{\cA}{\mathcal{A}}
\newcommand{\whcA}{\widehat{\cA}}

\newcommand{\cB}{\mathcal{B}}

\newcommand{\E}{\mathbb{E}}
\newcommand{\cE}{\mathcal{E}}

\newcommand{\fG}{\mathfrak{G}}

\newcommand{\cH}{\mathcal{H}}

\newcommand{\N}{\mathbb{N}}

\newcommand{\cN}{\mathcal{N}}

\newcommand{\cP}{\mathcal{P}}

\newcommand{\cQ}{\mathcal{Q}}

\newcommand{\R}{\mathbb{R}}

\newcommand{\cS}{\mathcal{S}}

\newcommand{\bbS}{\mathbb{S}}

\newcommand{\wtT}{\widetilde{T}}

\newcommand{\cU}{\mathcal{U}}

\newcommand{\cX}{\mathcal{X}}

\newcommand{\tPhi}{\widetilde{\Phi}}
\newcommand{\ttheta}{\widetilde{\theta}}

\newcommand{\tspan}{\mathrm{span}}

\newcommand{\tT}{\mathrm{T}}

\newcommand{\FWER}{\textsf{\textsc{fwer}}}
\newcommand{\FDP}{\textsf{\textsc{fdp}}}
\newcommand{\NDP}{\textsf{\textsc{ndp}}}
\newcommand{\FDR}{\textsf{\textsc{fdr}}}

\newcommand{\Dh}{D} 
\newcommand{\ds}{d} 
\newcommand{\Ns}{N} 
\newcommand{\As}{\cA} 
\newcommand{\ns}{n} 
\newcommand{\whAs}{\widehat{\As}} 

\theoremstyle{plain}
\newtheorem{lemma}{Lemma}
\newtheorem{theorem}{Theorem}

\newtheorem{proposition}{Proposition}

\theoremstyle{definition}
\newtheorem{definition}{Definition}
\theoremstyle{remark}
\newtheorem{remark}{Remark}

\newcommand{\revise}[1]{#1}

\begin{document}

\title{A Multiple Hypothesis Testing Approach to Low-Complexity Subspace Unmixing}

\author{Waheed U. Bajwa and Dustin G. Mixon
\thanks{Preliminary versions of some of the results reported in this paper were presented at the $50$th Annual Allerton Conference on Communication, Control, and Computing, Monticello, IL, Oct. 1--5, 2012 \cite{Bajwa.Mixon.Conf2012}. WUB is with the Department of Electrical and Computer Engineering, Rutgers, The State University of New Jersey, Piscataway, NJ 08854 (Email: {\tt waheed.bajwa@rutgers.edu}). DGM is with the Department of Mathematics and Statistics, Air Force Institute of Technology, Dayton, OH 45433 (Email: {\tt dustin.mixon@afit.edu}).}\thanks{The research of WUB is supported in part by the National Science Foundation under grant CCF-1218942 and by the Army Research Office under grant W911NF-14-1-0295. The research of DGM is supported in part by the AFOSR Young Investigator Research Program award, by the National Science Foundation under grant DMS-1321779, and by the Air Force Office of Scientific Research under grant F4FGA05076J002.}}


\maketitle

\begin{abstract}
\revise{Subspace-based signal processing traditionally focuses} on problems involving a few subspaces. Recently, a number of problems in different application areas have emerged that involve a significantly larger number of subspaces relative to the ambient dimension. It becomes imperative in such settings to first identify a smaller set of \emph{active subspaces} that contribute to the \revise{observation} before further processing can be carried out. This problem of identification of a small set of active subspaces among a huge collection of subspaces from a single (noisy) observation in the ambient space is termed \emph{subspace unmixing}. This paper formally poses the subspace unmixing problem \revise{under the \emph{parsimonious subspace-sum} (PS3) model, discusses connections of the PS3 model to} problems in wireless communications, hyperspectral imaging, high-dimensional statistics and compressed sensing, and proposes a low-complexity algorithm, termed \emph{marginal subspace detection} (MSD), for subspace unmixing. The MSD algorithm turns the subspace unmixing problem \revise{for the PS3 model} into a multiple hypothesis testing (MHT) problem and its analysis in the paper helps control the family-wise error rate of this MHT problem at any level $\alpha \in [0,1]$ \revise{under two random signal generation models}. Some other highlights of the analysis of the MSD algorithm include: ($i$) it is applicable to an arbitrary collection of subspaces on the Grassmann manifold; ($ii$) it relies on properties of the collection of subspaces that are computable in polynomial time; and ($iii$) it allows for linear scaling of the number of active subspaces as a function of the ambient dimension. Finally, numerical results are presented in the paper to better understand the performance of
the MSD algorithm.
\end{abstract}


\section{Introduction}
Subspace models, in which it is assumed signals of interest lie on or near a low-dimensional subspace of a higher-dimensional Hilbert space $\cH$, have a rich history in signal processing, machine learning, and statistics. While much of the classical literature in detection, estimation, classification, dimensionality reduction, etc., is based on the subspace model, many of these results deal with a small number of subspaces, say, $\cX_\Ns := \{\cS_1,\cS_2,\dots,\cS_\Ns\}$ with each $\cS_i$ a subspace of $\cH$, relative to the dimension of the Hilbert space: $\text{dim}(\cH) := \Dh \geq \Ns$. Consider, for instance, the classical subspace detection problem studied in \cite{Scharf.Friedlander.ITSP1994}. In this problem, one deals with two subspaces---the \emph{signal} subspace and the \emph{interference} subspace---and a solution to the detection problem involves a low-complexity generalized likelihood ratio test \cite{Scharf.Friedlander.ITSP1994}. However, proliferation of cheap sensors and low-cost semiconductor devices in the modern world means we often find ourselves dealing with a significantly larger number of subspaces relative to the extrinsic dimension, i.e., $\Dh \ll \Ns$. But many of the classical subspace-based results do not generalize in such ``$\Dh$ smaller than $\Ns$'' settings either because of the breakdown of the stated assumptions or because of the prohibitive complexity of the resulting solutions. In fact, without additional constraints, information processing in such settings might well be a daunting, if not impossible, task.

One constraint that often comes to our rescue in this regard in many applications is the ``principle of parsimony'': \emph{while the total number of subspaces might be large, only a small number of them, say, $\ns \propto \Dh$, tend to be ``active'' at any given instance}. \revise{Our focus in this paper is on the \emph{parsimonious subspace-sum} (PS3) model, in which case the $\Dh$-dimensional observation $y \in \cH$ can be mathematically expressed as $y \in \sum_{i \in \As} \cS_i + \text{noise}$, where $\As := \{i : \cS_i \in \cX_\Ns \text{ is active}\}$ denotes the set of indices of active subspaces with $\ns := |\As| \ll \Dh \ll \Ns$. Since finite-dimensional Hilbert spaces are isometrically isomorphic to Euclidean spaces, we will assume without loss of generality in the following that $\cH = \R^\Dh$.}\footnote{\revise{Note that all results presented in this paper can be extended in a straightforward manner to the case of a complex Hilbert space.}} It is easy to convince oneself in this case that the classical subspace-based computational machinery for information processing becomes available to us \revise{under the PS3 model} as soon as we have access to \revise{$\As$}. One of the fundamental challenges for information processing in the ``$\Dh$ smaller than $\Ns$'' setting \revise{under the PS3 model} could then be described as the recovery of the set of \revise{indices of} active subspaces, \revise{$\As \subset \{1,\dots,\Ns\}$}, from the $\Dh$-dimensional \revise{observation $y \in \sum_{i \in \As} \cS_i + \text{noise}$}. We term this problem of the recovery of \revise{$\As$ from a single noisy observation} as \emph{subspace unmixing}. Before describing our main contributions in relation to subspace unmixing \revise{under the PS3 model}, we discuss some of its applications in different areas.

\subsubsection{Multiuser Detection in Wireless Networks}
Consider a wireless network comprising a large number of users in which some of the users simultaneously transmit data to a base station. It is imperative for the base station in this case to identify the users that are communicating with it at any given instance, which is termed as the problem of multiuser detection. This problem of multiuser detection in wireless networks can also be posed as a subspace unmixing problem \revise{under the PS3 model}. In this context, users in the network communicate with the base station using $\Dh$-dimensional codewords in \revise{$\R^\Dh$}, each individual user is assigned a codebook that spans a low-dimensional subspace $\cS_i$ of \revise{$\R^\Dh$}, the total number of users in the network is $\Ns$, the number of active users at any given instance is $\ns \ll \Ns$, and the base station receives $y \in \sum_{i \in \As} \cS_i + \text{noise}$ due to the superposition property of the wireless medium, where $\As$ denotes the indices of the users actively communicating with the base station.

\subsubsection{Spectral Unmixing in Hyperspectral Remote Sensing}
Hyperspectral remote sensing has a number of civilian and defense applications, which typically involve identifying remote objects from their spectral signatures. Because of the low spatial resolution of hyperspectral imaging systems in most of these applications, individual hyperspectral pixels tend to comprise multiple objects (e.g., soil and vegetation). Spectral unmixing is the problem of decomposition of a ``mixed'' hyperspectral pixel into its constituent objects. In order to pose this spectral unmixing problem in terms of the subspace unmixing problem studied in this paper, we need two assumptions that are often invoked in the literature. First, the spectral variability of each object in different scenes can be captured through a low-dimensional subspace. Second, the mixture of spectra of different objects into a hyperspectral pixel can be described by a linear model. The spectral unmixing problem under these assumptions is the subspace unmixing problem \revise{under the PS3 model}, with \revise{$y \in \R^\Dh$} denoting the $\Dh$-dimensional hyperspectral pixel of an imaging system with $\Dh$ spectral bands, \revise{$\{\cS_i \subset \R^\Dh\}_{i=1}^{\Ns}$} denoting the low-dimensional subspaces of \revise{$\R^\Dh$} associated with the spectra of individual objects, $\Ns$ denoting the total number of objects of interest, and $y \in \sum_{i \in \As} \cS_i + \text{noise}$ with $\ns := |\As| \ll \Ns$ since only a small number of objects are expected to contribute to a single hyperspectral pixel.

\subsubsection{Group Model Selection in High-Dimensional Statistics}
Model selection in statistical data analysis is the problem of learning the relationship between the samples of a dependent or response variable (e.g., the malignancy of a tumor, the health of a network) and the samples of independent or predictor variables (e.g., the expression data of genes, the traffic data in the network). There exist many applications in statistical model selection where the implication of a single predictor in the response variable implies presence of other related predictors in the true model. In such situations, the problem of model selection is often reformulated in a ``group'' setting. This problem of group model selection in high-dimensional settings, where the number of predictors tends to be much larger than the number of samples, can also be posed as the subspace unmixing problem \revise{under the PS3 model}. In this context, \revise{$y \in \R^\Dh$} denotes the $\Dh$-dimensional response variable with $\Dh$ representing the total number of samples, $\Ns$ denotes the total number of groups of predictors that comprise the design matrix, \revise{$\{\cS_i \subset \R^\Dh\}_{i=1}^{\Ns}$} denotes the low-dimensional subspaces of \revise{$\R^\Dh$} spanned by each of the groups of predictors, and $y \in \sum_{i \in \As} \cS_i + \text{noise}$ with $\As$ denoting the indices of the groups of predictors that truly affect the response variable.

\subsubsection{Sparsity Pattern Recovery in Block-Sparse Compressed Sensing}
Compressed sensing is an alternative sampling paradigm for signals that have sparse representations in some orthonormal bases. In recent years, the canonical compressed sensing theory has been extended to the case of signals that have block-sparse representations in some orthonormal bases. Sparsity pattern recovery in block-sparse compressed sensing is the problem of identifying the nonzero ``block coefficients'' of the measured signal. The problem of sparsity pattern recovery in block-sparse compressed sensing, however, can also be posed as the subspace unmixing problem \revise{under the PS3 model}. In this context, \revise{$y \in \R^\Dh$} denotes the $\Dh$-dimensional measurement vector with $\Dh$ being the total number of measurements, $\Ns$ denotes the total number of blocks of coefficients, \revise{$\{\cS_i \subset \R^\Dh\}_{i=1}^{\Ns}$} denotes the low-dimensional subspaces of \revise{$\R^\Dh$} spanned by the ``blocks of columns'' of the composite matrix $\Phi\Psi$ with $\Phi$ being the measurement matrix and $\Psi$ being the sparsifying basis, and $y \in \sum_{i \in \As} \cS_i + \text{noise}$ with $\As$ denoting the indices of the nonzero blocks of coefficients of the signal in $\Psi$.

\subsection{Relationship to Prior Work}\label{ssec:prior_work}
Since the subspace unmixing problem \revise{under the PS3 model} has connections to a number of application areas, it is invariably related to prior works in some of those areas. In the context of multiuser detection, the work that is most closely related to ours is \cite{Applebaum.etal.PC2012}. However, the setup of \cite{Applebaum.etal.PC2012} can be considered a restrictive version of the \revise{two random signal generation models considered in here}. Roughly speaking, the \revise{signal generation model} in \cite{Applebaum.etal.PC2012} can be described as a \revise{\emph{randomly-modulated PS3 model}}, $y \in \sum_{i \in \As} \varepsilon_i \cS_i + \text{noise}$ with $\{\varepsilon_i\}_{i=1}^\Ns$ being independent and identically distributed isotropic random variables. In addition, the results of \cite{Applebaum.etal.PC2012} \revise{do not allow for an explicit control of the \emph{family-wise error rate} (\FWER) and also} rely on parameters that cannot be easily translated into properties of the subspaces alone. Finally, \cite{Applebaum.etal.PC2012} relies on a convex optimization procedure for multiuser detection that has superlinear (in $\Dh$ and $\Ns$) computational complexity.

In the context of group model selection and block-sparse compressed sensing, our work can be considered related to\revise{\cite{Yuan.Lin.JRSSSB2006,Bach.JMLR2008,Nardi.Rinaldo.EJS2008,Huang.Zhang.AS2010,Eldar.etal.ITSP2010,Ben-Haim.Eldar.IJSTSP2011,Elhamifar.Vidal.ITSP2012,Cotter.etal.ITSP2005,Tropp.etal.SP2006,Tropp.SP2006,Gribonval.etal.JFAA2008,Stojnic.etal.ITSP2009,Eldar.Rauhut.ITIT2010,Obozinski.etal.AS2011,Davies.Eldar.ITIT2012,Bajwa.etal.ITIT2015}}. None of these works, however, help us understand the problem of subspace unmixing \revise{under the PS3 model} in its general form. Some of these works, when translated into the general subspace unmixing problem \revise{under the PS3 model}, consider only random subspaces\revise{\cite{Bach.JMLR2008,Nardi.Rinaldo.EJS2008,Stojnic.etal.ITSP2009,Huang.Zhang.AS2010}}, study subspaces generated through a Kronecker operation \cite{Davies.Eldar.ITIT2012,Obozinski.etal.AS2011,Eldar.Rauhut.ITIT2010,Tropp.etal.SP2006,Tropp.SP2006,Cotter.etal.ITSP2005,Gribonval.etal.JFAA2008}, \revise{or ignore additive noise in the observation \cite{Bajwa.etal.ITIT2015}. Some other works that do not translate into $\cX_\Ns$ being a collection of random/Kronecker-structured subspaces suggest that, fixing the dimensions of subspaces, the total number of active subspaces can at best scale as $O\left(\sqrt{\Dh}\right)$ \cite{Huang.Zhang.AS2010,Eldar.etal.ITSP2010,Ben-Haim.Eldar.IJSTSP2011,Elhamifar.Vidal.ITSP2012}---the so-called ``square-root bottleneck.'' Further,} many of these works either focus on computational approaches that have superlinear complexity\revise{\cite{Obozinski.etal.AS2011,Yuan.Lin.JRSSSB2006,Bach.JMLR2008,Huang.Zhang.AS2010,Elhamifar.Vidal.ITSP2012,Nardi.Rinaldo.EJS2008,Tropp.SP2006,Bajwa.etal.ITIT2015}} or suggest that low-complexity approaches suffer from the ``dynamic range of active subspaces'' \cite{Gribonval.etal.JFAA2008,Ben-Haim.Eldar.IJSTSP2011}. \revise{Finally, none of these works help control the \FWER~of the subspace unmixing problem.}

\revise{We conclude this discussion by pointing out that the subspace unmixing problem under the PS3 model is effectively a solved problem for the case of one-dimensional subspaces ($\ds = 1$). Notable works in this regard that neither consider random subspaces nor suffer from the square-root bottleneck include \cite{Tropp.ACHA2008,Tropp.CRASSI2008,Candes.Plan.AS2009,Bajwa.etal.JCN2010,Bajwa.etal.ACHA2012}. Among these works, \cite{Tropp.ACHA2008,Tropp.CRASSI2008,Candes.Plan.AS2009} focus on computational approaches with superlinear complexity and do not facilitate control of the \FWER, while \cite{Bajwa.etal.JCN2010,Bajwa.etal.ACHA2012} analyze a low-complexity approach. Despite the fact that \cite{Tropp.ACHA2008,Tropp.CRASSI2008,Candes.Plan.AS2009,Bajwa.etal.JCN2010,Bajwa.etal.ACHA2012} do not explicitly address the subspace unmixing problem, one of the main insights offered by these works is that the square-root bottleneck in high-dimensional problems can often be broken through the use of appropriate random signal models. We leverage this insight in the following and rely on two random signal generation models for our analysis that can be considered natural generalizations of the ones in \cite{Tropp.ACHA2008,Tropp.CRASSI2008,Candes.Plan.AS2009,Bajwa.etal.JCN2010,Bajwa.etal.ACHA2012} for multi-dimensional ($\ds > 1$) subspaces.}

\subsection{Our Contributions}
The main contributions of this paper are as follows. First, it formally puts forth the problem of subspace unmixing \revise{under the PS3 model} that provides a mathematically unified view of many problems studied in other application areas. Second, it presents a low-complexity solution to the problem of \revise{unmixing under the PS3 model} that has linear complexity in $\Dh$, $\Ns$, and the dimensions of the individual subspaces. Third, it presents \revise{comprehensive analyses} of the proposed solution, termed \emph{marginal subspace detection} (MSD), \revise{under two random signal generation models that, while assuming the contributions of different subspaces to the observation to be random, do not require the subspaces themselves to be random. In particular, the resulting analyses rely on geometric measures that can be explicitly computed in polynomial time and provide means of controlling the \FWER~of the subspace unmixing problem at any level $\alpha \in [0,1]$. Finally, the analyses under both signal generation models neither suffer from the square-root bottleneck nor get affected by the dynamic range of the active subspaces.} We conclude by pointing out that a preliminary version of this work appeared in \cite{Bajwa.Mixon.Conf2012}. However, that work was focused primarily on group model selection, it did not account for noise in the \revise{observation}, and the ensuing analysis lacked details in terms of the metrics of multiple hypothesis testing.

\subsection{Notation and Organization}
The following notational convention is used throughout the rest of this paper. We use the standard notation $:=$ to denote definitions of terms. The notation $|\cdot|$ is used for both the cardinality of a set and the absolute value of a real number. Similarly, $\|\cdot\|_2$ is used for both the $\ell_2$-norm of a vector and the operator 2-norm of a matrix. The notation $\setminus$ denotes the set difference operation. Finally, we make use of the following ``\emph{Big--O}'' notation for scaling relations: $f(n) = O(g(n))$ if $\exists c_o > 0, n_o : \forall n \geq n_o, f(n) \leq c_o g(n)$, $f(n) = \Omega(g(n))$ (alternatively, $f(n) \succeq g(n)$) if $g(n) = O(f(n))$, and $f(n) =  \Theta(g(n))$ if $g(n) = O(f(n))$ and $f(n) = O(g(n))$.

The rest of this paper is organized as follows. In Sec.~\ref{sec:prob_form}, we formulate the problem of subspace unmixing \revise{under the PS3 model, put forth the two random signal generation models studied in this paper, define the relevant metrics used to measure the performance of subspace unmixing algorithms, and introduce different geometric measures of the collection of subspaces involved in the subspace unmixing problem}. In Sec.~\ref{sec:MSD}, we describe our proposed algorithm for subspace unmixing \revise{under the PS3 model. In Sec.~\ref{sec:MSD_FMB}, we provide an analysis of the proposed algorithm under one of the random signal generation models and discuss the significance of our results in the context of related results in the literature on group model selection and block-sparse compressed sensing. In Sec.~\ref{sec:geometry}, we extend the analysis in Sec.~\ref{sec:MSD_FMB} to provide the most general results for unmixing under the PS3 model.} In Sec.~\ref{sec:num_res}, we present some numerical results to support our analyses and we finally conclude in Sec.~\ref{sec:conc}.

\section{Problem Formulation}\label{sec:prob_form}
\revise{Consider the $\Dh$-dimensional Euclidean space $\R^\Dh$ and the Grassmann manifold $\fG(\ds, \Dh)$, which denotes the collection of all $\ds$-dimensional subspaces of $\R^\Dh$.} Next, consider a collection of $\Ns \gg \Dh/\ds \gg 1$ subspaces given by $\cX_\Ns = \big\{\cS_i \in \fG(\ds,\Dh), i=1,\dots,\Ns\big\}$ such that $\cS_1,\dots,\cS_\Ns$ are pairwise disjoint: $\cS_i \cap \cS_j = \{0\}~\forall i,j=1,\dots,\Ns, i \not= j$. Heuristically, this means each of the subspaces in $\cX_\Ns$ is low-dimensional and, collectively, the subspaces can potentially ``fill'' the ambient space $\R^\Dh$. The fundamental assumptions in the problem of subspace unmixing \revise{under the parsimonious subspace-sum (PS3) model considered} in this paper are that only a small number $\ns < \Dh/\ds \ll \Ns$ of the subspaces are active at any given instance and the \revise{observation} $y \in \R^\Dh$ corresponds to a noisy version of an $x \in \R^\Dh$ that lies in the sum of the active subspaces. Mathematically, we can formalize these assumptions by defining $\As = \{i : \cS_i \in \cX_\Ns \text{ is active}\}$, writing $x \in \sum_{i \in \As} \cS_i$, and stating that the \revise{observation} $y = x + \eta$, where $\eta \in \R^\Dh$ denotes noise in the \revise{observation}. For the sake of this exposition, we assume $\eta$ to be either bounded energy, deterministic error, i.e., $\|\eta\|_2 < \epsilon_\eta$, or independent and identically distributed (i.i.d.) Gaussian noise with variance $\sigma^2$, i.e., $\eta \sim \cN(0, \sigma^2 I)$.

The final detail we need in order to complete formulation of the problem of subspace unmixing is a mathematical model for generation of the ``noiseless signal'' $x \in \sum_{i \in \As} \cS_i$. \revise{In this regard, we first assume the following probabilistic model for the \emph{activity pattern} of the underlying subspaces:}
\begin{itemize}
\item \emph{\textbf{Random Activity Pattern:}} The set of indices of the active subspaces $\As$ is a random $\ns$-subset of $\{1,\dots,\Ns\}$ with $\Pr(\As = \{i_1,i_2,\dots,i_\ns\}) = 1/\binom{\Ns}{\ns}$.
\end{itemize}
\revise{Next, we state the most-general generative model, termed \emph{random directions model}, for $x$ studied in this paper.
\begin{itemize}
\item \emph{\textbf{Random Directions Model:}} Conditioned on the random activity pattern $\As = \{i_1,i_2,\dots,i_\ns\}$, the noiseless signal $x$ can be expressed as $x := \sum_{j=1}^{\ns} x_{i_j}$. Next, define an $n$-tuple of (unit) direction vectors as $$\mathfrak{X}^n := \big(x_{i_1}/\|x_{i_1}\|_2,\dots,x_{i_n}/\|x_{i_n}\|_2\big) \in \mathfrak{B}^n := (\bbS^{\Dh-1} \cap \cS_{i_1}) \times \dots \times (\bbS^{\Dh-1} \cap \cS_{i_n}),$$ where $\bbS^{\Dh-1}$ denotes the unit sphere in $\R^\Dh$. Then, $\mathfrak{X}^n$ is drawn independently of $\As$ from $\mathfrak{B}^n$ according to a product probability measure $\lambda_{\mathfrak{B}^n}$ on $\mathfrak{B}^n$; that is, for all Borel sets $B^n \subset \mathfrak{B}^n$, we have $$\Pr(\mathfrak{X^n} \in B^n|\cA) = \Pr(\mathfrak{X^n} \in B^n) = \lambda_{\mathfrak{B}^n}(B^n).$$
\end{itemize}}
Given this \revise{random directions} generative model, the goal of subspace unmixing in this paper is to identify the set of \revise{indices of active subspaces $\As$} using knowledge of the collection of subspaces $\cX_\Ns$ and the noisy \revise{observation} $y \in \R^\Dh$. In particular, our focus is on unmixing solutions with linear (in $\ds$, $\Ns$, and $\Dh$) computational complexity.

A few remarks are in order now regarding the stated assumptions \revise{and signal generation model}. First, the assumption of pairwise disjointness of the subspaces is much weaker than the assumption of linear independence of the subspaces, which is typically invoked in the literature on subspace-based information processing \cite{Scharf.Friedlander.ITSP1994,Manolakis.etal.ITGRS2001}.\footnote{The other commonly invoked assumption of orthogonal subspaces is of course impossible in the $\Dh/\ds \ll \Ns$ setting.} In particular, while pairwise disjointness implies pairwise linear independence, it does not preclude the possibility of an element in one subspace being representable through a linear combination of elements in two or more subspaces. Second, the rather mild assumption on the randomness of the activity pattern can be interpreted as the lack of a priori information concerning the activity pattern of subspaces. \revise{Third, unlike works such as \cite{Bach.JMLR2008,Nardi.Rinaldo.EJS2008,Stojnic.etal.ITSP2009,Huang.Zhang.AS2010} in the literature on group model selection and block-sparse compressed sensing, the random directions model does not assume that the collection of subspaces $\cX_\Ns$ are drawn randomly from $\fG(\ds, \Dh)$. Rather, $\cX_\Ns$ can be any arbitrary (random or deterministic) collection of subspaces and the model makes a significantly weaker assumption that the contributions of active subspaces to the observation $y$ point in random directions that are independent of the indices of active subspaces. It is worth noting here that the random directions model is one of the key reasons our analysis will be able to break the square-root bottleneck for arbitrary collections of subspaces that satisfy certain geometric properties (cf.~Sec.~\ref{ssec:FMB_SQRT} and Sec.~\ref{sec:geometry}). And while the motivation for this model comes from the existing literature on compressed sensing and model selection \cite{Tropp.ACHA2008,Tropp.CRASSI2008,Candes.Plan.AS2009,Bajwa.etal.ITIT2015}, the algorithmic and analytical approaches used in here as well as the nature of the final results are fundamentally different from earlier works. In particular, while our work allows $\lambda_{\mathfrak{B}^n}$ to be any arbitrary product probability measure, prior works such as \cite{Tropp.ACHA2008,Tropp.CRASSI2008,Candes.Plan.AS2009,Bajwa.etal.ITIT2015} provide results for a significantly restrictive class of  product probability measures.}

\revise{Although the random directions model is adequate for the problem of unmixing under the PS3 model, our forthcoming analysis will also require the description of an alternative generative model for the noiseless signal $x \in \sum_{i \in \As} \cS_i$. The purpose of the alternative model, which we term \emph{fixed mixing bases model}, is twofold. First, it will turn out that results derived under the (seemingly restrictive) fixed mixing bases model can be generalized for the random directions model in a straightforward manner (cf.~Sec.~\ref{sec:geometry}). Second, despite the somewhat specialized nature of the fixed mixing bases model, it does arise explicitly in application areas such as group model selection and block-sparse compressed sensing in which the contribution of each subspace is explicitly representable using a fixed orthonormal basis. Formally, the fixed mixing bases model has the following description.
\begin{itemize}
\item \emph{\textbf{Fixed Mixing Bases Model:}} Each subspace $\cS_i$ in the collection $\cX_\Ns$ is associated with an orthonormal basis $\Phi_i \in \R^{\Dh \times \ds}$, i.e., $\tspan(\Phi_i) = \cS_i$ and $\Phi_i^\tT \Phi_i = I$. Further, there is a deterministic but unknown collection of ``mixing coefficients'' $\{\theta_j \in \R^\ds, j=1,\dots,\ns\}$ such that the noiseless signal $x$ is given by $x := \sum_{j=1}^\ns x_{i_j}$ with $x_{i_j} := \Phi_{i_j} \theta_j \in \cS_{i_j}$, where the random activity pattern $\As = \{i_1,i_2,\dots,i_\ns\}$.
\end{itemize}
Readers familiar with detection under the classical linear model \cite[Sec.~7.7]{Kay.Book1998} will recognize the assumption $x = \sum_{j=1}^\ns \Phi_{i_j} \theta_j$ as a simple generalization of that setup for the problem of subspace unmixing. Notice that unlike the random directions model, the fixed mixing bases model has no randomness associated with the contribution $x_{i_j}$ of each active subspace, which is a relaxation of related assumptions in the literature on model selection and compressed sensing \cite{Tropp.ACHA2008,Tropp.CRASSI2008,Candes.Plan.AS2009,Bajwa.etal.ITIT2015}. On the other hand, unlike the fixed mixing bases model, the random directions model is completely agnostic to the representation of the contribution $x_{i_j}$ of each active subspace to the observation $y$.}

\subsection{Performance Metrics}\label{ssec:perf_metrics}
In this paper, we address the problem of subspace unmixing \revise{under the PS3 model} by transforming it into a multiple hypothesis testing problem (cf.~Sec.~\ref{sec:MSD}). While several measures of error have been used over the years in multiple hypothesis testing problems, the two most widely accepted ones in the literature remain the \emph{family-wise error rate} (\FWER) and the \emph{false discovery rate} (\FDR) \cite{Farcomeni.SMiMR2008}. Mathematically, if we use $\whAs \subset \{1,\dots,\Ns\}$ to denote an estimate of the indices of active subspaces returned by an unmixing algorithm then controlling the \FWER~at level $\alpha$ in our setting means $\FWER := \Pr(\whAs \not\subset \As) \leq \alpha$. In words, $\FWER \leq \alpha$ guarantees that the probability of declaring even one inactive subspace as active (i.e., a single \emph{false positive}) is controlled at level $\alpha$. On the other hand, controlling the \FDR~in our setting controls the \emph{expected proportion} of inactive subspaces that are incorrectly declared as active by an unmixing algorithm \cite{Benjamini.Hochberg.JRSSSB1995}.

While the \FDR~control is less stringent than the \FWER~control \cite{Benjamini.Hochberg.JRSSSB1995}, \revise{our goal in this paper is control of the \FWER~under both signal generation models}. This is because control of the \FDR~in the case of dependent test statistics, which will be the case in our setting (cf.~Sec.~\ref{sec:MSD}), is a challenging research problem \cite{Benjamini.etal.B2006}. Finally, once we control the \FWER~at some level $\alpha$, our goal is to have as large a fraction of active subspaces identified as active by the unmixing algorithm as possible. The results reported in the paper in this context will be given in terms of the \emph{non-discovery proportion} (\NDP), defined as $\NDP := \frac{|\As \setminus \whAs|}{|\As|}$.

\subsection{Preliminaries}\label{ssec:prelim}
In this section, we introduce some definitions that will be used throughout the rest of this paper to characterize the performance of our proposed approach to subspace unmixing \revise{under both the random directions model and the fixed mixing bases model}. It is not too difficult to convince oneself that the ``hardness'' of subspace unmixing problem should be a function of the ``similarity'' of the underlying subspaces: \emph{the more similar the subspaces in $\cX_\Ns$, the more difficult it should be to tell them apart}. In order to capture this intuition, we work with the similarity measure of \emph{subspace coherence} in this paper, defined as:
\begin{align}
\gamma(\cS_i, \cS_j) := \max_{w \in \cS_i, z \in \cS_j} \frac{|\langle w,z \rangle|}{\|w\|_2\|z\|_2},
\end{align}
where $(\cS_i,\cS_j)$ denote two $\ds$-dimensional subspaces in $\R^\Dh$. Note that $\gamma : \fG(\ds, \Dh) \times \fG(\ds, \Dh) \rightarrow [0,1]$ simply measures cosine of the smallest principal angle between two subspaces and has appeared in earlier literature \cite{Drmac.SJMAA2000,Elhamifar.Vidal.ITSP2012}. In particular, given (any arbitrary) orthonormal bases $U_i$ and $U_j$ of $\cS_i$ and $\cS_j$, respectively, it follows that $\gamma(\cS_i, \cS_j) := \|U_i^\tT U_j\|_2$. Since we are interested in unmixing \emph{any} active collection of subspaces, we will be stating our main results in terms of the \emph{local $2$-subspace coherence} \revise{and the \emph{quadratic-mean subspace coherence}} of individual subspaces, defined in the following.
\begin{definition}[Local $2$-Subspace Coherence]
Given a collection of subspaces $\cX_\Ns = \big\{\cS_i \in \fG(\ds,\Dh), i=1,\dots,\Ns\big\}$, the local $2$-subspace coherence of subspace $\cS_i$ is defined as $\gamma_{2,i} := \max_{j \not=i ,k \not= i: j \not= k} \big[\gamma(\cS_i, \cS_j) + \gamma(\cS_i, \cS_k)\big]$.
\end{definition}
\begin{definition}[\revise{Quadratic-Mean Subspace Coherence}]
\revise{Given a collection of subspaces $\cX_\Ns = \big\{\cS_i \in \fG(\ds,\Dh), i=1,\dots,\Ns\big\}$, the quadratic-mean subspace coherence of subspace $\cS_i$ is defined as $\gamma_{\textsf{rms},i} := \sqrt{\frac{1}{\Ns-1}\sum_{j\not=i} \gamma^2(\cS_i, \cS_j)}$.}
\end{definition}

In words, $\gamma_{2,i}$ measures closeness of $\cS_i$ to the worst pair of subspaces in the collection $\cX^{-i}_\Ns := \cX_\Ns \setminus \{\cS_i\}$\revise{, while $\gamma_{\textsf{rms},i}$ measures its closeness to the entire collection of subspaces in $\cX^{-i}_\Ns$ in terms of the quadratic mean. Note that $\gamma_{\textsf{rms},i}$ is a generalization of the \emph{mean square coherence} defined in \cite{Barg.etal.ITIT2015} to the case of multi-dimensional subspaces}. It follows from the definition of subspace coherence that $\gamma_{2,i} \in [0,2]$ \revise{and $\gamma_{\textsf{rms},i} \in [0,1]$}, with $\gamma_{2,i} \revise{= \gamma_{\textsf{rms},i} = 0}$ if and only if every subspace in $\cX^{-i}_\Ns$ is orthogonal to $\cS_i$, while $\gamma_{2,i} = 2$ \revise{(resp., $\gamma_{\textsf{rms},i} = 1$)} if and only if two \revise{(resp., all)} subspaces in $\cX^{-i}_\Ns$ are the same as $\cS_i$. Because of our assumption of pairwise disjointness, however, we have that $\gamma_{2,i}$ \revise{(resp., $\gamma_{\textsf{rms},i}$)} is strictly less than $2$ \revise{(resp., 1)} in this paper. We conclude our discussion of the local $2$-subspace coherence \revise{and the quadratic-mean subspace coherence} by noting that \revise{both of them are} trivially computable in polynomial time.

\begin{remark}
\revise{Since $\ns$ subspaces contribute to the observation, it is natural to ask whether one should utilize some measure of \emph{local $(\ns-1)$-subspace coherence} in lieu of $\gamma_{2,i}$ and $\gamma_{\textsf{rms},i}$ to analyze the problem of subspace unmixing; see, e.g, the related notion of \emph{cumulative coherence} for $\ds =1$ in the literature on compressed sensing \cite{Tropp.ITIT2004}. While this is a valid line of reasoning, measures such as local $(\ns-1)$-subspace coherence cannot be explicitly computed in polynomial time. In contrast, $\gamma_{2,i}$ and $\gamma_{\textsf{rms},i}$ are not only easily computable, but their use in our analysis also allows for linear scaling of the number of active subspaces for appropriate collections of subspaces.}
\end{remark}

The next definition we need to characterize the performance of subspace unmixing is \emph{active subspace energy}.
\begin{definition}[Active Subspace Energy]
Given the set of indices of active subspaces $\As = \{i_1,i_2,\dots,i_\ns\}$ and the noiseless signal \revise{$x = \sum_{j=1}^\ns x_{i_j}$ with $x_{i_j} \in \cS_{i_j}$}, the energy of the $i_j$-th active subspace is defined as \revise{$\cE_{i_j} := \|x_{i_j}\|_2^2$}.
\end{definition}
\noindent \revise{In the case of the fixed mixing bases model, notice that $\cE_{i_j} \equiv \|\theta_j\|_2^2$ due to the orthonormal nature of the mixing bases.} Inactive subspaces of course contribute no energy to the \revise{observation}, i.e., $\cE_i = 0~\forall i \not\in \As$. But it is important for us to specify the energy of active subspaces for subspace unmixing. Indeed, active subspaces that contribute too little energy to the final \revise{observation} to the extent that they get buried in noise cannot be identified using any computational method.

\revise{Next, the low-complexity algorithm proposed in this paper requires an additional definition of \emph{cumulative active subspace energy} to characterize its unmixing performance under the PS3 model.}
\begin{definition}[Cumulative Active Subspace Energy]
Given the set of indices of active subspaces $\As$, the cumulative active subspace energy is defined as $\cE_\As := \sum_{i \in \As} \cE_i$.
\end{definition}
\noindent In words, cumulative active subspace energy can be considered a measure of ``signal energy'' and together with the noise energy/variance, it characterizes signal-to-noise ratio for the subspace unmixing problem.

\revise{Finally, we also need the definition of \emph{average mixing coherence} of individual subspaces for the analysis of our proposed unmixing algorithm under the fixed mixing bases model.}
\begin{definition}[Average Mixing Coherence]
Given a collection of subspaces $\cX_\Ns = \big\{\cS_i \in \fG(\ds,\Dh), i=1,\dots,\Ns\big\}$ and the associated mixing bases $\cB_\Ns := \big\{\Phi_i: \tspan(\Phi_i) = \cS_i, \Phi_i^\tT \Phi_i = I, i=1,\dots,\Ns\big\}$ \revise{under the fixed mixing bases model}, the average mixing coherence of subspace $\cS_i$ is defined as $\rho_i := \frac{1}{\Ns-1}\left\|\sum_{j\not=i} \Phi_i^\tT \Phi_j\right\|_2$.
\end{definition}
\noindent \revise{In words, average mixing coherence measures the ``niceness'' of the mixing bases in relation to each of the subspaces in the collection $\cX_\Ns$.} Since we are introducing average mixing coherence for the first time in the literature,\footnote{We refer the reader to our preliminary work \cite{Bajwa.Mixon.Conf2012} and a recent work \cite{Calderbank.etal.ACHA2015} for a related concept of \emph{average group/block coherence}.} it is worth understanding its behavior. First, unlike (local $2$-)subspace coherence, it is not invariant to the choice of the (mixing) bases. \revise{While this suggests it won't be useful for analysis of the general subspace unmixing problem, we will later see in Sec.~\ref{sec:geometry} that the average mixing coherence is intricately related to the quadratic-mean subspace coherence under the random directions model.} Second, note that $\rho_i \in [0,1]$. To see this, observe that $\rho_i = 0$ if the subspaces in $\cX_\Ns$ are orthogonal to each other. Further, we have from triangle inequality and the definition of subspace coherence that $\rho_i \leq \sum_{j\not=i} \gamma(\cS_i,\cS_j)/(\Ns-1) \leq 1$. Clearly, the \emph{average subspace coherence} of the subspace $\cS_i$, defined as $\overline{\gamma}_i := \sum_{j\not=i} \gamma(\cS_i,\cS_j)/(\Ns-1)$, is a trivial upper bound on $\rho_i$. 
We conclude by noting that the average mixing coherence, $\rho_i$, is trivially computable in polynomial time \revise{under the fixed mixing bases model}.

\section{Marginal Subspace Detection for Subspace Unmixing}\label{sec:MSD}
\revise{We present our low-complexity approach to subspace unmixing in this section, while its performance under the two random signal generation models introduced in Sec.~\ref{sec:prob_form} will be characterized in the following sections.} Recall that the \revise{observation} $y \in \R^\Dh$ \revise{is} given by $y = x + \eta$ with $x \in \sum_{i \in \As} \cS_i$. Assuming the cardinality of the set of indices of active subspaces, $\ns = |\As|$, is known, one can pose the subspace unmixing problem as an $M$-ary hypothesis testing problem with $M = \binom{\Ns}{\ns}$. In this formulation, we have that the $k$-th hypothesis, $\cH_k,\,k=1,\dots,M$, corresponds to one of the $M$ possible choices for the set $\As$. While an optimal theoretical strategy in this setting will be to derive the $M$-ary maximum likelihood decision rule, this will lead to superlinear computational complexity since one will have to evaluate $M = \binom{\Ns}{\ns} \succeq \left(\frac{\Ns}{\ns}\right)^\ns$ test statistics, one for each of the $M$ hypotheses, in this formulation. Instead, since we are interested in low-complexity approaches in this paper, we approach the problem of subspace unmixing as $\Ns$ individual binary hypothesis testing problems. An immediate benefit of this approach, which transforms the problem of subspace unmixing into a multiple hypothesis testing problem, is the computational complexity: \emph{we need only evaluate $\Ns$ test statistics in this setting}. The challenges in this setting of course are specifying the decision rules for each of the $\Ns$ binary hypotheses and understanding the performance of the corresponding low-complexity approach in terms of the \FWER~and the \NDP. We address \revise{the first challenge by describing a matched subspace detector-based multiple hypothesis testing approach to subspace unmixing in the following, while the second challenge will be addressed for the fixed mixing bases model and the random directions model in Sec.~\ref{sec:MSD_FMB} and Sec.~\ref{sec:geometry}, respectively.}

In order to solve the problem of subspace unmixing, we propose to work with $\Ns$ binary hypothesis tests on the \revise{observation} $y = x + \eta$, as defined below. \revise{
\begin{align}
    \cH^k_0 \ &: \ x \in \sum_{j=1}^\ns \cS_{i_j} \quad \text{s.t.} \quad k \not\in \As = \{i_1,i_2,\dots,i_\ns\} , \quad k=1,\dots,\Ns,\\
    \cH^k_1 \ &: \ x \in \sum_{j=1}^\ns \cS_{i_j} \quad \text{s.t.} \quad k \in \As = \{i_1,i_2,\dots,i_\ns\} , \quad k=1,\dots,\Ns.
\end{align}
}
In words, the null hypothesis $\cH^k_0$ being true signifies that subspace $\cS_k$ is not active, while the alternative hypothesis $\cH^k_1$ being true signifies that $\cS_k$ is active. Note that if we fix a $k \in \{1,\dots,\Ns\}$ then deciding between $\cH^k_0$ and $\cH^k_1$ is equivalent to detecting a subspace $\cS_k$ in the presence of an interference signal and additive noise. While this setup is reminiscent of the subspace detection problem studied in \cite{Scharf.Friedlander.ITSP1994}, the fundamental differences between the binary hypothesis test(s) in our problem and that in \cite{Scharf.Friedlander.ITSP1994} are that: ($i$) the interference signal in \cite{Scharf.Friedlander.ITSP1994} is assumed to come from a \emph{single}, known subspace, while the interference signal in our problem setup is a function of the underlying activity pattern of the subspaces; and ($ii$) our problem setup involves multiple hypothesis tests (with dependent test statistics), which therefore requires control of the \FWER. Nonetheless, since \emph{matched subspace detectors} are known to be (quasi-)optimal in subspace detection problems \cite{Scharf.Friedlander.ITSP1994}, we put forth the test statistics for our $\Ns$ binary hypothesis tests that are based on matched subspace detectors.

\begin{algorithm*}[t]
\caption{Marginal Subspace Detection (MSD) for Subspace Unmixing}
\label{alg:MSD}
\textbf{Input:} Collection $\cX_\Ns = \big\{\cS_i \in \fG(\ds,\Dh), i=1,\dots,\Ns\big\}$, \revise{observation} $y \in \R^\Dh$, and thresholds $\big\{\tau_i > 0\big\}_{i=1}^\Ns$\\
\textbf{Output:} An estimate $\whcA \subset \{1,\dots,\Ns\}$ of the set of indices of active subspaces
\begin{algorithmic}
\STATE $U_k \leftarrow \text{An orthonormal basis of the subspace } \cS_k, \ k=1,\dots,\Ns$ \hfill \COMMENT{Initialization}%
\STATE $T_k(y) \leftarrow \|U_k^\tT y\|_2^2, \ k=1,\dots,\Ns$ \hfill
\COMMENT{Computation of test statistics}%
\STATE $\whcA \leftarrow \big\{k \in \{1,\dots,\Ns\}: T_k(y) > \tau_k\big\}$ \hfill \COMMENT{Decision rules for marginal detection}
\end{algorithmic}
\end{algorithm*}

Specifically, in order to decide between $\cH^k_0$ and $\cH^k_1$ for any given $k$, we compute the test statistic $T_k(y) := \|U_k^T y\|_2^2$, where $U_k$ denotes any orthonormal basis of the subspace $\cS_k$. Notice that $T_k(y)$ is invariant to the choice of the basis $U_k$ and therefore it can be computed \revise{irrespective of whether $x$ is generated under the random directions model or the fixed mixing bases model}. In order to relate this test statistic to the classical subspace detection literature, note that $T_k(y) = \|U_k U_k^T y\|_2^2 = \|\cP_{\cS_k} y\|_2^2$. That is, the test statistic is equivalent to projecting the \revise{observation} onto the subspace $\cS_k$ and computing the energy of the projected \revise{observation}, which is the same operation that arises in the classical subspace detection literature \cite{Scharf.Friedlander.ITSP1994,Schaum.Conf2001}. The final decision between $\cH^k_0$ and $\cH^k_1$ then involves comparing the test statistic against a threshold $\tau_k$:
\begin{align}
    T_k(y) \underset{\cH^k_0}{\overset{\cH^k_1}{\gtrless}} \tau_k, \quad k=1,\dots,\Ns.
\end{align}
Once we obtain these marginal decisions, we can use them to obtain an estimate of the set of indices of the active subspaces by setting $\whcA = \{k : \cH^k_1 \text{ is accepted}\}$. We term this entire procedure, outlined in Algorithm~\ref{alg:MSD}, as \emph{marginal subspace detection} (MSD) because of its reliance on detecting the presence of subspaces in the active set using marginal test statistics. The challenge then is understanding the behavior of the test statistics for each subspace under the two hypotheses and specifying values of the thresholds $\{\tau_k\}$ that lead to acceptable \FWER~and \NDP~figures \revise{for the two random signal generation models under consideration}. Further, a key aspect of any analysis of MSD involves understanding the number of active subspaces that can be tolerated by it as a function of the subspace collection $\cX_\Ns$, the ambient dimension $\Dh$, the subspace dimension $\ds$, etc. In order to address these questions, one would ideally like to understand the distributions of the test statistics for each of the $\Ns$ subspaces under the two different hypotheses. However, specifying these distributions under the \revise{two signal generation models} of Sec.~\ref{sec:prob_form} \emph{and} ensuring that ($i$) the final results can be interpreted in terms of the geometry of the underlying subspaces, and ($ii$) the number of active subspaces can be allowed to be almost linear in $\frac{\Dh}{\ds}$ appears to be an intractable problem. \revise{Therefore, we will instead focus} on characterizing the (right and left) tail probabilities \revise{(i.e., $\Pr\left(T_k(y) \geq \tau \big| \cH^k_0\right)$ and $\Pr\left(T_k(y) \leq \tau \big| \cH^k_1\right)$)} of the test statistics for each subspace \revise{under the two random signal generation models}.

\section{\revise{Performance of Marginal Subspace Detection Under the Fixed Mixing Bases Model}}\label{sec:MSD_FMB}%
\revise{Our goal in this section is performance characterization of MSD under the assumption of $x$ being generated using the fixed mixing bases model. Interestingly, we will later see in Sec.~\ref{sec:geometry} that the results derived in this setting can be generalized to the case of $x$ being generated using the random directions model in a straightforward manner.}

\revise{We begin with an evaluation of $\Pr\left(T_k(y) \geq \tau \big| \cH^k_0\right)$, which will help control the \FWER~of MSD at a prescribed level $\alpha$.} To this end, we assume an arbitrary (but fixed) $k \in \{1,\dots,\Ns\}$ in the following and derive the right-tail probability under the null hypothesis, i.e., \revise{$y = \sum_{j=1}^\ns x_{i_j} + \eta = \sum_{j=1}^\ns \Phi_{i_j} \theta_j + \eta$ and $k \not\in \As = \{i_1,i_2,\dots,i_\ns\}$, where the $\Phi_i$'s denote the \emph{fixed} mixing bases}. In order to facilitate the forthcoming analysis, we note that since $T_k(y)$ is invariant to the choice of $U_k$, we have $T_k(y) = \big\|\sum_{j=1}^\ns U_k^\tT\Phi_{i_j} \theta_j + U_k^\tT\eta\big\|^2_2 \equiv \big\|\sum_{j=1}^\ns \Phi_k^\tT\Phi_{i_j} \theta_j + \Phi_k^\tT\eta\big\|^2_2$. We now state the result that characterizes the right-tail probability of $T_k(y)$ under the null hypothesis, $\cH_0^k$.
\begin{lemma}\label{lemma:right_tail_null_hypo}
Under the null hypothesis $\cH_0^k$ for any fixed $k \in \{1,\dots,\Ns\}$, the test statistic has the following right-tail probability \revise{for the fixed mixing bases model}:
\begin{enumerate}
\item In the case of bounded deterministic error $\eta$ and the assumption $\tau > (\epsilon_\eta + \rho_k \sqrt{\ns \cE_\As})^2$, we have
\begin{align}
    \Pr\left(T_k(y) \geq \tau \big| \cH^k_0\right) \leq e^2 \exp\left(-\frac{c_0(\Ns-\ns)^2\big(\sqrt{\tau} - \epsilon_\eta - \rho_k \sqrt{\ns
        \cE_\As}\big)^2}{\Ns^2 \gamma^2_{2,k} \cE_\As}\right).
\end{align}

\item In the case of i.i.d. Gaussian noise $\eta$, define $\epsilon := \sigma\sqrt{\ds + 2\delta + 2\sqrt{\ds\delta}}$ for any $\delta > 0$. Then, under the assumption $\tau > (\epsilon + \rho_k \sqrt{\ns \cE_\As})^2$, we have
\begin{align}
    \Pr\left(T_k(y) \geq \tau \big| \cH^k_0\right) \leq e^2 \exp\left(-\frac{c_0(\Ns-\ns)^2\big(\sqrt{\tau} - \epsilon - \rho_k \sqrt{\ns \cE_\As}\big)^2}{\Ns^2 \gamma^2_{2,k} \cE_\As}\right) + \exp(-\delta).
\end{align}
\end{enumerate}
Here, the parameter $c_0 := \frac{e^{-1}}{256}$ is an absolute positive constant.
\end{lemma}

\revise{The proof of this lemma is given in Appendix~\ref{app:lemma1}.} Our next goal is evaluation of $\Pr\left(T_k(y) \leq \tau \big| \cH^k_1\right)$\revise{, which will help understand the \NDP~performance of MSD under the fixed mixing bases model when its \FWER~is controlled at level $\alpha$.} In this regard, we once again fix an arbitrary $k \in \{1,\dots,\Ns\}$ and derive the left-tail probability under the alternative hypothesis, $\cH^k_1$, i.e., $y = \sum_{j=1}^\ns \Phi_{i_j} \theta_j + \eta$ such that the index $k \in \As = \{i_1,i_2,\dots,i_\ns\}$.
\begin{lemma}\label{lemma:left_tail_alt_hypo}
Under the alternative hypothesis $\cH_1^k$ for any fixed $k \in
\{1,\dots,\Ns\}$, the test statistic has the following left-tail probability
\revise{for the fixed mixing bases model}:
\begin{enumerate}
\item In the case of bounded deterministic error $\eta$ and under the
    assumptions $\cE_k > (\epsilon_\eta + \rho_k \sqrt{\ns(\cE_\As -
\cE_k)})^2$ and $\tau < (\sqrt{\cE_k} - \epsilon_\eta - \rho_k
\sqrt{\ns(\cE_\As - \cE_k)})^2$, we have
\begin{align}
\label{eqn:lemma:left_tail_alt_hypo:1}
    \Pr\left(T_k(y) \leq \tau \big| \cH^k_1\right) \leq e^2 \exp\left(-\frac{c_0(\Ns-\ns)^2\big(\sqrt{\cE_k} - \sqrt{\tau} - \epsilon_\eta - \rho_k
        \sqrt{\ns(\cE_\As - \cE_k)}\big)^2}{(2\Ns - \ns)^2 \gamma^2_{2,k} (\cE_\As - \cE_k)}\right).
\end{align}

\item In the case of i.i.d. Gaussian noise $\eta$, define $\epsilon :=
    \sigma\sqrt{\ds + 2\delta + 2\sqrt{\ds\delta}}$ for any $\delta > 0$. Then,
    under the assumptions $\cE_k > (\epsilon + \rho_k \sqrt{\ns(\cE_\As
    - \cE_k)})^2$ and $\tau < (\sqrt{\cE_k} - \epsilon - \rho_k
    \sqrt{\ns(\cE_\As - \cE_k)})^2$, we have
\begin{align}
\label{eqn:lemma:left_tail_alt_hypo:2}
    \Pr\left(T_k(y) \leq \tau \big| \cH^k_1\right) \leq e^2 \exp\left(-\frac{c_0(\Ns-\ns)^2\big(\sqrt{\cE_k} - \sqrt{\tau} - \epsilon - \rho_k
        \sqrt{\ns(\cE_\As - \cE_k)}\big)^2}{(2\Ns - \ns)^2 \gamma^2_{2,k} (\cE_\As - \cE_k)}\right) + \exp(-\delta).
\end{align}
\end{enumerate}
Here, the parameter $c_0 := \frac{e^{-1}}{256}$ is an absolute positive
constant.
\end{lemma}

\revise{The proof of this lemma is provided in Appendix~\ref{app:lemma2}. Before proceeding with the implications of Lemmas~\ref{lemma:right_tail_null_hypo} and \ref{lemma:left_tail_alt_hypo} for the fixed mixing bases model, it is instructive to provide an intuitive interpretation of these lemmas} for individual subspaces (i.e., in the absence of a formal correction for multiple hypothesis testing \cite{Farcomeni.SMiMR2008,Benjamini.Hochberg.JRSSSB1995}). We provide such an interpretation in the following for the case of bounded deterministic error $\eta$, with the understanding that extensions of our arguments to the case of i.i.d. Gaussian noise $\eta$ are straightforward.

\subsection{\revise{Discussion of the Lemmata}}
Lemma~\ref{lemma:right_tail_null_hypo} characterizes the probability of \emph{individually} rejecting the null hypothesis $\cH_0^k$ when it is true \revise{under the fixed mixing bases model} (i.e., declaring the subspace $\cS_k$ to be active when it is inactive). Suppose for the sake of argument that $\cH_0^k$ is true and $\cS_k$ is orthogonal to every subspace in $\cX_\Ns \setminus \{\cS_k\}$, in which case the $k$-th test statistic reduces to $T_k(y) \equiv \|\eta\|_2^2$. It is then easy to see in this hypothetical setting that the decision threshold $\tau_k$ must be above the \emph{noise floor}, $\tau_k > \epsilon_\eta^2$, to ensure one does not reject $\cH_0^k$ when it is true. Lemma~\ref{lemma:right_tail_null_hypo} effectively generalizes this straightforward observation \revise{under the fixed mixing bases model} to the case when the $\cS_k$ cannot be orthogonal to every subspace in $\cX_\Ns \setminus \{\cS_k\}$. First, the lemma states in this case that an \emph{effective noise floor}, defined as $\epsilon^2_{\text{eff}} := (\epsilon_\eta + \rho_k \sqrt{\ns \cE_\As})^2$, appears in the problem and the decision threshold must now be above this effective noise floor, $\tau_k > \epsilon_{\text{eff}}^2$, to ensure one does not reject $\cH_0^k$ when it is true. It can be seen from the definition of the effective noise floor that $\epsilon_{\text{eff}}$ has an intuitive additive form, with the first term $\epsilon_\eta$ being due to the additive error $\eta$ and the second term $\rho_k \sqrt{\ns \cE_\As}$ being due to the mixing with non-orthogonal \revise{bases (subspaces)}. In particular, $\epsilon_{\text{eff}} \searrow \epsilon_\eta$ as the average mixing coherence $\rho_k \searrow 0$ (recall that $\rho_k \equiv 0$ for the case of $\cS_k$ being orthogonal to the subspaces in $\cX_\Ns \setminus \{\cS_k\}$). Second, once a threshold above the effective noise floor is chosen, the lemma states that the probability of rejecting the true $\cH_0^k$ decreases exponentially as the gap between the threshold and the effective noise floor increases and/or the local $2$-subspace coherence $\gamma_{2,k}$ of $\cS_k$ decreases. In particular, the probability of rejecting the true $\cH_0^k$ in this case has the intuitively pleasing characteristic that it approaches zero exponentially fast as $\gamma_{2,k} \searrow 0$ (recall that $\gamma_{2,k} \equiv 0$ for the case of $\cS_k$ being orthogonal to the subspaces in $\cX_\Ns \setminus \{\cS_k\}$).

We now shift our focus to Lemma~\ref{lemma:left_tail_alt_hypo}, which specifies the probability of individually rejecting the alternative hypothesis $\cH_1^k$ \revise{under the fixed mixing bases model} when it is true (i.e., declaring the subspace $\cS_k$ to be inactive when it is indeed active). It is once again instructive to first understand the hypothetical scenario of $\cS_k$ being orthogonal to every subspace in $\cX_\Ns \setminus \{\cS_k\}$. In this case, the $k$-th test statistic under $\cH_1^k$ being true reduces to \revise{$T_k(y) \equiv \|x_k + U_k^\tT \eta\|_2^2$}, where \revise{$x_k$} denotes the component of the noiseless signal $x$ that is contributed by the subspace $\cS_k$. Notice in this hypothetical setting that the rotated additive error $U_k^\tT \eta$ can in principle be antipodally aligned with the signal component \revise{$x_k$}, thereby reducing the value of $T_k(y)$. It is therefore easy to argue in this idealistic setup that ensuring one does accept $\cH_1^k$ when it is true requires: ($i$) the energy of the subspace $\cS_k$ to be above the \emph{noise floor}, $\cE_k > \epsilon_\eta^2$, so that the test statistic remains strictly positive; and ($ii$) the decision threshold $\tau_k$ to be \emph{below} the \emph{subspace-to-noise gap}, $\tau_k < (\sqrt{\cE_k} - \epsilon_\eta)^2$, so that the antipodal alignment of $U_k^\tT \eta$ with \revise{$x_k$} does not result in a false negative. We now return to the statement of Lemma~\ref{lemma:left_tail_alt_hypo} and note that it also effectively generalizes these straightforward observations \revise{under the fixed mixing bases model} to the case when the $\cS_k$ cannot be orthogonal to every subspace in $\cX_\Ns \setminus \{\cS_k\}$. First, similar to the case of Lemma~\ref{lemma:right_tail_null_hypo}, this lemma states in this case that an \emph{effective noise floor}, defined as $\epsilon^2_{\text{eff}} := (\epsilon_\eta + \rho_k \sqrt{\ns(\cE_\As - \cE_k)})^2$, appears in the problem and the energy of the subspace $\cS_k$ must now be above this effective noise floor, $\cE_k > \epsilon^2_{\text{eff}}$, to ensure that the test statistic remains strictly positive. In addition, we once again have an intuitive additive form of $\epsilon_{\text{eff}}$, with its first term being due to the additive error $\eta$, its second term being due to the mixing with non-orthogonal \revise{bases (subspaces)}, and $\epsilon_{\text{eff}} \searrow \epsilon_\eta$ as the average mixing coherence $\rho_k \searrow 0$. Second, the lemma states that the decision threshold must now be below the \emph{subspace-to-effective-noise gap}, $\tau_k < (\sqrt{\cE_k} - \epsilon_{\text{eff}})^2$. Third, once a threshold below the subspace-to-effective-noise gap is chosen, the lemma states that the probability of rejecting the true $\cH_1^k$ decreases exponentially as the gap between $(\sqrt{\cE_k} - \epsilon_{\text{eff}})^2$ and the threshold increases and/or the local $2$-subspace coherence $\gamma_{2,k}$ of $\cS_k$ decreases. In particular, Lemma~\ref{lemma:left_tail_alt_hypo} once again has the intuitively pleasing characteristic that the probability of rejecting the true $\cH_1^k$ approaches zero exponentially fast as $\gamma_{2,k} \searrow 0$.

\subsection{\revise{Main Results for the Fixed Mixing Bases Model}}
It can be seen from the preceding discussion that increasing the values of the decision thresholds $\{\tau_k\}$ in MSD should decrease the \FWER~\revise{under the fixing mixing bases model}. Such a decrease in the \FWER~of course will come at the expense of an increase in the \NDP. We will specify this relationship between the $\tau_k$'s and the \NDP~in the following. But we first characterize one possible choice of the $\tau_k$'s that helps control the \FWER~of MSD at a predetermined level $\alpha$ \revise{for the fixed mixing bases model}. The following theorem makes use of Lemma~\ref{lemma:right_tail_null_hypo} and the Bonferroni correction for multiple hypothesis testing \cite{Farcomeni.SMiMR2008}.

\begin{theorem}\label{thm:FWER_MSD}
The family-wise error rate of the marginal subspace detection (Algorithm~\ref{alg:MSD}) can be controlled at any level $\alpha \in [0,1]$ \revise{under the fixed mixing bases model} by selecting the decision thresholds $\{\tau_k\}_{k=1}^{\Ns}$ as follows:
\begin{enumerate}
\item In the case of bounded deterministic error $\eta$, select $$\tau_k = \left(\epsilon_\eta + \rho_k \sqrt{\ns \cE_\As} + \frac{\gamma_{2,k}\Ns}{\Ns-\ns}\sqrt{c_0^{-1}\cE_\As\log\big(\tfrac{e^2 \Ns}{\alpha}\big)}\right)^2, \quad k=1,\dots,\Ns.$$

\item In the case of i.i.d. Gaussian noise $\eta$, select $$\tau_k = \left(\sigma\sqrt{\ds + 2\log\big(\tfrac{2\Ns}{\alpha}\big) + 2\sqrt{\ds\log\big(\tfrac{2\Ns}{\alpha}\big)}} + \rho_k \sqrt{\ns \cE_\As} + \frac{\gamma_{2,k}\Ns}{\Ns-\ns}\sqrt{c_0^{-1}\cE_\As\log\big(\tfrac{e^2 2\Ns}{\alpha}\big)}\right)^2, \quad k=1,\dots,\Ns.$$
\end{enumerate}
\end{theorem}
\begin{proof}
The Bonferroni correction for multiple hypothesis testing dictates that the \FWER~of the MSD is guaranteed to be controlled at a level $\alpha \in [0,1]$ as long as the probability of false positive of each \emph{individual} hypothesis is controlled at level $\frac{\alpha}{\Ns}$ \cite{Farcomeni.SMiMR2008}, i.e., $\Pr\left(T_k(y) \geq \tau_k \big| \cH^k_0\right) \leq \frac{\alpha}{\Ns}$. The statement for the bounded deterministic error $\eta$ can now be shown to hold by plugging the prescribed decision thresholds into Lemma~\ref{lemma:right_tail_null_hypo}. Similarly, the statement for the i.i.d. Gaussian noise $\eta$ can be shown to hold by plugging $\delta := \log\big(\frac{2\Ns}{\alpha}\big)$ and the prescribed decision thresholds into Lemma~\ref{lemma:right_tail_null_hypo}.
\end{proof}

A few remarks are in order now regarding Theorem~\ref{thm:FWER_MSD}. We once again limit our discussion to the case of bounded deterministic error, since its extension to the case of i.i.d. Gaussian noise is straightforward. In the case of deterministic error $\eta$, Theorem~\ref{thm:FWER_MSD} requires the decision thresholds to be of the form $\tau_k = (\epsilon_\eta + \epsilon_{m,1} + \epsilon_{m,2})^2$, where $\epsilon_\eta$ captures the effects of the additive error, $\epsilon_{m,1}$ is due to the mixing with non-orthogonal \revise{bases}, and $\epsilon_{m,2}$ \revise{(which is invariant to the choice of the mixing bases)} captures the effects of both the mixing with non-orthogonal subspaces and the \FWER~$\alpha$.\footnote{In here, we are suppressing the dependence of $\epsilon_{m,1}$ and $\epsilon_{m,2}$ on the subspace index $k$ for ease of notation.} Other factors that affect the chosen thresholds \revise{under the fixed mixing bases model} include the total number of subspaces, the number of active subspaces, and the cumulative active subspace energy. But perhaps the most interesting aspect of Theorem~\ref{thm:FWER_MSD} is the fact that as the mixing \revise{bases/subspaces} become ``closer'' to being orthogonal, the chosen thresholds start approaching the noise floor $\epsilon_\eta^2$: $\tau_k \searrow \epsilon_\eta^2$ as $\rho_k, \gamma_{2,k} \searrow 0$.

While Theorem~\ref{thm:FWER_MSD} helps control the \FWER~of MSD \revise{under the fixed mixing bases model}, it does not shed light on the corresponding \NDP~figure for MSD. In order to completely characterize the performance of MSD \revise{for the fixed mixing bases model}, therefore, we also need the following theorem.
\begin{theorem}\label{thm:NDP_MSD}
Suppose the family-wise error rate of the marginal subspace detection (Algorithm~\ref{alg:MSD}) \revise{for the fixed mixing bases model} is controlled at level $\alpha \in [0,1]$ by selecting the decision thresholds $\{\tau_k\}_{k=1}^{\Ns}$ specified in Theorem~\ref{thm:FWER_MSD}. Then the estimate of the indices of active subspaces returned by MSD \revise{under the fixed mixing bases model} satisfies $\whcA \supset \As_*$ with probability exceeding $1 - \varepsilon$, where:
\begin{enumerate}
\item In the case of bounded deterministic error $\eta$, we have $\varepsilon := \Ns^{-1} + \alpha$ and $$\As_* := \left\{i \in \As : \cE_i > \left(2\epsilon_\eta + \rho_i \sqrt{\ns \cE_{1,i}} + \frac{\gamma_{2,i}\Ns}{\Ns-\ns}\sqrt{c_0^{-1}\cE_{2,i}}\right)^2\right\}$$ with parameters $\cE_{1,i} := \Big(\sqrt{\cE_\As} + \sqrt{\cE_\As - \cE_i}\Big)^2$ and $\cE_{2,i} := \Big(\sqrt{\cE_\As\log(\tfrac{e^2 \Ns}{\alpha})} + (2-\frac{\ns}{\Ns})\sqrt{2(\cE_\As - \cE_i)\log(eN)}\Big)^2$.

\item In the case of i.i.d. Gaussian noise $\eta$, we have $\varepsilon := \Ns^{-1} + \frac{3}{2}\alpha$ and $$\As_* := \left\{i \in \As : \cE_i > \left(2\epsilon + \rho_i \sqrt{\ns \cE_{1,i}} + \frac{\gamma_{2,i}\Ns}{\Ns-\ns}\sqrt{c_0^{-1}\cE_{2,i}}\right)^2\right\}$$ with the three parameters $\epsilon := \sigma\sqrt{\ds + 2\log\big(\tfrac{2\Ns}{\alpha}\big) + 2\sqrt{\ds\log\big(\tfrac{2\Ns}{\alpha}\big)}}$, $\cE_{1,i} := \Big(\sqrt{\cE_\As} + \sqrt{\cE_\As - \cE_i}\Big)^2$ and $\cE_{2,i} := \Big(\sqrt{\cE_\As\log(\tfrac{e^2 2 \Ns}{\alpha})} + (2-\frac{\ns}{\Ns})\sqrt{2(\cE_\As - \cE_i)\log(eN)}\Big)^2$.
\end{enumerate}
\end{theorem}
\begin{proof}
In order to prove the statement for the bounded deterministic error $\eta$, pick an arbitrary $i \in \As_*$ and notice that the assumptions within Lemma~\ref{lemma:left_tail_alt_hypo} for the subspace $\cS_i \in \cX_\Ns$ are satisfied by virtue of the definition of $\As_*$ and the choice of the decision thresholds in Theorem~\ref{thm:FWER_MSD}. It therefore follows from \eqref{eqn:lemma:left_tail_alt_hypo:1} in Lemma~\ref{lemma:left_tail_alt_hypo} that $i \not\in \whcA$ with probability at most $\Ns^{-2}$. We can therefore conclude by a simple union bound argument that $\As_* \not\subset \whcA$ with probability at most $\Ns^{-1}$. The statement now follows from a final union bound over the events $\As_* \not\subset \whcA$ and $\whAs \not\subset \As$, where the second event is needed since we are \emph{simultaneously} controlling the \FWER~at level $\alpha$. Likewise, the statement for the i.i.d. Gaussian noise $\eta$ can be shown to hold by first plugging $\delta := \log\big(\frac{2\Ns}{\alpha}\big)$ into \eqref{eqn:lemma:left_tail_alt_hypo:2} in Lemma~\ref{lemma:left_tail_alt_hypo} and then making use of similar union bound arguments.
\end{proof}
\begin{remark}
An astute reader will notice that we are being loose in our union bounds for the case of i.i.d. Gaussian noise. Indeed, we are double counting the event that the sum of squares of $\ds$ i.i.d. Gaussian random variables exceeds $\epsilon^2$, once during Lemma~\ref{lemma:right_tail_null_hypo} (which is used for \FWER~calculations) and once during Lemma~\ref{lemma:left_tail_alt_hypo} (which is used for this theorem). In fact, it can be shown through a better bookkeeping of probability events that $\varepsilon = \Ns^{-1} + \alpha$ for i.i.d. Gaussian noise also. Nonetheless, we prefer the stated theorem because of the simplicity of its proof.
\end{remark}

It can be seen from Theorem~\ref{thm:NDP_MSD} that if one controls the \FWER~of the MSD using Theorem~\ref{thm:FWER_MSD} then its \NDP~figure \revise{for the fixed mixing bases model} satisfies $\NDP \leq \frac{|\As \setminus \As_*|}{\ns}$ with probability exceeding $1 - \Ns^{-1} - \Theta(\alpha)$. Since $\As_* \subset \As$, it then follows that the \NDP~figure is the smallest when the cardinality of $\As_*$ is the largest. It is therefore instructive to understand the nature of $\As_*$ \revise{under the fixed mixing bases model}, which is the set of indices of active subspaces that are guaranteed to be identified as active by the MSD algorithm. Theorem~\ref{thm:NDP_MSD} tells us that \emph{any} active subspace whose energy is not ``too small'' is a member of $\As_*$ \revise{under the fixed mixing bases model}. Specifically, in the case of bounded deterministic error, the threshold that determines whether the energy of an active subspace is large or small for the purposes of identification by MSD takes the form $(2\epsilon_\eta + \tilde{\epsilon}_{m,1} + \tilde{\epsilon}_{m,2})^2$. Here, similar to the case of Theorem~\ref{thm:FWER_MSD}, we observe that $\tilde{\epsilon}_{m,1}$ and $\tilde{\epsilon}_{m,2}$ are \emph{pseudo-noise terms} that appear \emph{only} due to the mixing with non-orthogonal \revise{bases/subspaces} and that depend upon additional factors such as the total number of subspaces, the number of active subspaces, the cumulative active subspace energy, and the \FWER.\footnote{We are once again suppressing the dependence of $\tilde{\epsilon}_{m,1}$ and $\tilde{\epsilon}_{m,2}$ on the subspace index for ease of notation.} In particular, we once again have the intuitive result that $\tilde{\epsilon}_{m,1}, \tilde{\epsilon}_{m,2} \searrow 0$ as $\rho_i, \gamma_{2,i} \searrow 0$, implying that any active subspace whose energy is on the order of the noise floor will be declared as active by the MSD algorithm in this setting. Since this is the best that any subspace unmixing algorithm can be expected to accomplish, one can argue that the MSD algorithm \revise{under the fixed mixing bases model} performs near-optimal subspace unmixing for the case when the average mixing coherences and the local $2$-subspace coherences of individual subspaces in the collection $\cX_\Ns$ are significantly small. Finally, note that this intuitive understanding of MSD can be easily extended to the case of i.i.d. Gaussian noise, with the major difference being that $\epsilon_\eta$ in that case gets replaced by $\epsilon = \sigma\sqrt{\ds + 2\log\big(\tfrac{2\Ns}{\alpha}\big) + 2\sqrt{\ds\log\big(\tfrac{2\Ns}{\alpha}\big)}}$.

\subsection{\revise{Breaking the Square-Root Bottleneck}}\label{ssec:FMB_SQRT}
Theorem~\ref{thm:FWER_MSD} establishes that the \FWER~of MSD \revise{under the fixed mixing bases model} can be controlled at any level $\alpha \in [0,1]$ through appropriate selection of the decision thresholds. Further, Theorem~\ref{thm:NDP_MSD} shows that the selected thresholds enable the MSD algorithm to identify all active subspaces whose energies exceed \emph{effective} noise floors characterized by additive error/noise, average mixing coherences, local $2$-subspace coherences, etc. Most importantly, these effective noise floors approach the ``true'' noise floor as the average mixing coherences and the local $2$-subspace coherences of individual \revise{bases/subspaces} approach zero, suggesting near-optimal nature of MSD for such collections of mixing subspaces in the ``$\Dh$ smaller than $\Ns$'' setting. But we have presented no mathematical evidence to suggest the average mixing coherences and the local $2$-subspace coherences of individual \revise{bases/subspaces} can indeed be small enough for the effective noise floors of Theorem~\ref{thm:NDP_MSD} to be on the order of $\big(\text{true noise floor} + o(1)\big)$. Our \revise{primary} goal in this section is to provide evidence to this effect by arguing for the existence of collection of \revise{bases/subspaces} whose average mixing coherences and local $2$-subspace coherences approach zero at significantly fast rates. \revise{But in the process, we also make an important observation in the context of group model selection and block-sparsity pattern recovery, namely, \emph{it is possible to break the square-root bottleneck in such problems without resorting to either random/Kronecker-structured or one-dimensional subspaces} (cf.~Sec.~\ref{ssec:prior_work} and Sec.~\ref{sec:prob_form}).}
\begin{remark}
\label{rem:sq_root_bottleneck}
\revise{Note that an approach is said to break the square-root bottleneck as long as it allows $\ns \ds = \Omega(\Dh^\varrho)$ with $\varrho > 1/2$ for \emph{some} collections of subspaces; see, e.g., \cite{Tropp.ACHA2008,Tropp.CRASSI2008,Kuppinger.etal.ITIT2012} for one-dimensional subspaces. Prior to this work, however, there existed no results that could be translated into such a guarantee for \emph{any} given collection of (non-random) multi-dimensional subspaces in the $\Dh$ smaller than $\Ns$ setting.}
\end{remark}

Recall from the statement of Theorem~\ref{thm:NDP_MSD} and the subsequent discussion that the effective noise floor for the $i$-th subspace involves additive pseudo-noise terms of the form
\begin{align}
\label{eqn:effective_noise_term}
    \epsilon^i_f := \rho_i \sqrt{\ns \cE_{1,i}} + \frac{\gamma_{2,i}\Ns}{\Ns-\ns}\sqrt{c_0^{-1}\cE_{2,i}},
\end{align}
where $\sqrt{\cE_{1,i}} = \Theta\Big(\sqrt{\cE_\As}\Big)$ and $\sqrt{\cE_{2,i}} = \Theta\Big(\sqrt{\cE_\As \log(\Ns/\alpha)}\Big)$. Since we are assuming that the number of active subspaces $\ns = O(\Ns)$, it follows that $\epsilon^i_f = \Theta\Big(\rho_i \sqrt{\ns \cE_\As}\Big) + \Theta\Big(\gamma_{2,i} \sqrt{\cE_\As \log(\Ns/\alpha)}\Big)$. In order to ensure $\epsilon^i_f = o(1)$, therefore, we need the following two conditions to hold \revise{under the fixed mixing bases model}:
\begin{align}
\label{eqn:mixing_coh_cond}
    \rho_i &= O\left(\frac{1}{\sqrt{\ns \cE_\As}}\right), \quad \text{and}\\
\label{eqn:local_2_coh_cond}
    \gamma_{2,i} &= O\left(\frac{1}{\sqrt{\cE_\As \log(\Ns/\alpha)}}\right).
\end{align}
Together, we term the conditions \eqref{eqn:mixing_coh_cond} and \eqref{eqn:local_2_coh_cond} as \emph{subspace coherence conditions}. Both these conditions are effectively statements about the geometry of the mixing subspaces and the corresponding mixing bases. In order to understand the implications of these two conditions, we parameterize the cumulative active subspace energy as $\cE_\As = \Theta(\ns^\delta)$ for $\delta \in [0,1]$. Here, $\delta = 0$ corresponds to one extreme of the cumulative active subspace energy staying constant as the number of active subspaces increases, while $\delta = 1$ corresponds to other extreme of the cumulative active subspace energy increasing linearly with the number of active subspaces.

We now turn our attention to the extreme of $\delta = 1$, in which case the subspace coherence conditions reduce to $\rho_i = O(\ns^{-1})$ and $\gamma_{2,i} = O(\ns^{-1/2}\log^{-1/2}(\Ns/\alpha))$. We are interested in this setting in understanding whether there indeed exist subspaces and mixing bases that satisfy these conditions. We have the following theorem in this regard, which also sheds light on the maximum number of active subspaces that can be tolerated by the MSD algorithm \revise{under the fixed mixing bases model}.
\begin{theorem}\label{thm:lin_scaling}
Suppose the number of active subspaces satisfies $\ns \leq \min\left\{\sqrt{\Ns} - 1,\frac{c_1^2 \Dh(\Ns - 1)}{(\Ns\ds - \Dh)\log(\Ns/\alpha)}\right\}$ for some constant $c_1 \in (0,1)$. Then there exist collections of subspaces $\cX_\Ns = \big\{\cS_i \in \fG(\ds,\Dh), i=1,\dots,\Ns\big\}$ and corresponding mixing bases $\cB_\Ns = \big\{\Phi_i: \tspan(\Phi_i) = \cS_i, \Phi_i^\tT \Phi_i = I, i=1,\dots,\Ns\big\}$ such that $\rho_i \leq \ns^{-1}$ and $\gamma_{2,i} \leq c_2 \ns^{-1/2}\log^{-1/2}(\Ns/\alpha)$ for $i=1,\dots,\Ns$, where $c_2 \geq \max\{2c_1,1\}$ is a positive numerical constant.
\end{theorem}
\begin{proof}
The proof of this theorem follows from a combination of results reported in \cite{Calderbank.etal.ACHA2015}. To begin, note from the definition of local $2$-subspace coherence that $\frac{\gamma_{2,i}}{2} \leq \mu(\cX_\Ns) := \max_{i\not=j} \gamma(\cS_i, \cS_j)$. We now argue there exist $\cX_\Ns$'s such that $\mu(\cX_\Ns) = 0.5 c_2 \ns^{-1/2}\log^{-1/2}(\Ns/\alpha)$, which in turn implies $\gamma_{2,i} \leq c_2 \ns^{-1/2}\log^{-1/2}(\Ns/\alpha)$ for such collections of subspaces. The quantity $\mu(\cX_\Ns)$, termed \emph{worst-case subspace coherence}, has been investigated extensively in the literature \cite{Lemmens.Seidel.IMP1973,Calderbank.etal.ACHA2015}. The first thing we need to be careful about is the fact from \cite[Th.~3.6]{Lemmens.Seidel.IMP1973}\cite[Th.~2.3]{Calderbank.etal.ACHA2015} that $\mu(\cX_\Ns) \geq \sqrt{\frac{\Ns\ds - \Dh}{\Dh(\Ns-1)}}$, which is ensured by the conditions $\ns \leq \frac{c_1^2 \Dh(\Ns - 1)}{(\Ns\ds - \Dh)\log(\Ns/\alpha)}$ and $c_2 \geq 2c_1$. The existence of such collections of subspaces now follows from \cite{Calderbank.etal.ACHA2015}, which establishes that the worst-case subspace coherences of many collections of subspaces (including subspaces drawn uniformly at random from $\fG(\ds,\Dh)$) come very close to meeting the lower bound $\sqrt{\frac{\Ns\ds - \Dh}{\Dh(\Ns-1)}}$.

In order to complete the proof, we next need to establish that if a collection of subspaces has $\mu(\cX_\Ns) = 0.5 c_2 \ns^{-1/2}\log^{-1/2}(\Ns/\alpha)$ then there exists \emph{at least} one corresponding mixing bases for that collection such that $\rho_i \leq \ns^{-1}$. In this regard, note that $\rho_i \leq \nu(\cB_\Ns) := \max_i \rho_i$. The quantity $\nu(\cB_\Ns)$, termed \emph{average group/block coherence}, was introduced in \cite{Bajwa.Mixon.Conf2012} and investigated further in \cite{Calderbank.etal.ACHA2015}. In particular, it follows from \cite[Lemma~3.4]{Calderbank.etal.ACHA2015} that every collection of subspaces $\cX_\Ns$ has at least one mixing bases with $\nu(\cB_\Ns) \leq \frac{\sqrt{\Ns}+1}{\Ns-1}$, which can in turn be upper bounded by $\ns^{-1}$ for $\ns \leq \sqrt{\Ns} - 1$.
\end{proof}

Recall that our problem formulation calls for $\ns < \Dh/\ds \ll \Ns$. Theorem~\ref{thm:lin_scaling} helps quantify these inequalities \revise{under the fixed mixing bases model} for the case of linear scaling of cumulative active subspace energy. Specifically, note that $\frac{\Dh(\Ns - 1)}{(\Ns\ds - \Dh)\log(\Ns/\alpha)} = O\left(\frac{\Dh}{\ds\log(\Ns/\alpha)}\right)$ for large $\Ns$. We therefore have that Theorem~\ref{thm:lin_scaling} allows the number of active subspaces to scale linearly with the extrinsic dimension $\Dh$ modulo a logarithmic factor. Stated differently, Theorem~\ref{thm:lin_scaling} establishes that the total number of \emph{active dimensions}, $\ns\ds$, can be proportional to the extrinsic dimension $\Dh$, while the total number of subspaces in the collection, $\Ns$, affect the number of active dimensions only through a logarithmic factor. Combining Theorem~\ref{thm:lin_scaling} with the earlier discussion, therefore, one can conclude that the MSD algorithm \revise{under the fixed mixing bases model} does not suffer from the ``square-root bottleneck'' of $\ns\ds = O(\sqrt{\Dh})$ despite the fact that its performance is being characterized in terms of polynomial-time computable measures. \revise{This is in stark contrast to related results in \cite{Huang.Zhang.AS2010,Eldar.etal.ITSP2010,Ben-Haim.Eldar.IJSTSP2011,Elhamifar.Vidal.ITSP2012} on group model selection and block-sparsity pattern recovery, which do not allow for linear scaling of the number of active dimensions in \emph{any} setting due to the fundamental limit $\mu(\cX_\Ns) \geq \sqrt{\frac{\Ns\ds - \Dh}{\Dh(\Ns-1)}}$~(cf.~Remark~\ref{rem:sq_root_bottleneck})}. Finally, we note that the constraint $\ns = O(\sqrt{\Ns})$ in Theorem~\ref{thm:lin_scaling} appears due to our use of \cite[Lemma~3.4]{Calderbank.etal.ACHA2015}, which not only guarantees existence of appropriate mixing bases but also provides a polynomial-time algorithm for obtaining those mixing bases. If one were interested in merely proving existence of \revise{``good''} mixing bases then this condition can be relaxed to $\ns = O(\Ns)$ by making use of \cite[Th.~3.2]{Calderbank.etal.ACHA2015} instead in the proof.

Since Theorem~\ref{thm:lin_scaling} guarantees existence of subspaces and mixing bases that satisfy the subspace coherence conditions for $\delta = 1$, it also guarantees the same for any other sublinear scaling $(0 \leq \delta < 1)$ of cumulative active subspace energy. Indeed, as $\delta \searrow 0$, the subspace coherence conditions (cf.~\eqref{eqn:mixing_coh_cond} and \eqref{eqn:local_2_coh_cond}) only become more relaxed. In fact, it turns out that the order-wise performance of the MSD algorithm no longer remains a function of the mixing bases for certain collections of subspaces when cumulative active subspace energy reaches the other extreme of $\delta = 0$. This assertion follows from the following theorem and the fact that $\delta = 0$ reduces the subspace coherence conditions to $\rho_i = O(\ns^{-1/2})$ and $\gamma_{2,i} = O(\log^{-1/2}(\Ns/\alpha))$.
\begin{theorem}\label{thm:constant_energy} Suppose the number of active subspaces satisfies $\ns \leq \frac{c_3 \Dh(\Ns - 1)}{\Ns\ds - \Dh}$ for some constant $c_3 \in (0,1)$ and the total number of subspaces in the collection $\cX_\Ns$ satisfies $\Ns \leq \alpha \exp(\ns/4)$. In such cases, there exist collections of subspaces that satisfy $\mu(\cX_\Ns) := \max_{i\not=j} \gamma(\cS_i, \cS_j) \leq \ns^{-1/2}$. Further, all such collections satisfy $\rho_i \leq \ns^{-1/2}$ and $\gamma_{2,i} \leq \log^{-1/2}(\Ns/\alpha)$ for $i=1,\dots,\Ns$. \end{theorem} \begin{proof} The proof of this theorem also mainly follows from \cite{Calderbank.etal.ACHA2015}, which establishes that there exist many collections of subspaces for which $\mu(\cX_\Ns) = \sqrt{\frac{\Ns\ds - \Dh}{c_3 \Dh(\Ns-1)}}$ for appropriate constants $c_3 \in (0,1)$. Under the condition $\ns \leq \frac{c_3 \Dh(\Ns - 1)}{\Ns\ds - \Dh}$, therefore, it follows that $\mu(\cX_\Ns) \leq \ns^{-1/2}$. Since $\gamma_{2,i} \leq 2 \mu(\cX_\Ns)$, we in turn obtain $\gamma_{2,i} \leq \log^{-1/2}(\Ns/\alpha)$ under the condition $\Ns \leq \alpha \exp(\ns/4)$. Finally, we have from the definition of the average mixing coherence that $\rho_i \leq \mu(\cX_\Ns)$, which in turn implies $\rho_i \leq \ns^{-1/2}$ and this completes the proof of the theorem. \end{proof}

Once again, notice that Theorem~\ref{thm:constant_energy} allows linear scaling of the number of active dimensions as a function of the extrinsic dimension. In words, Theorem~\ref{thm:constant_energy} tells us that MSD \revise{can be used} for unmixing of collections of subspaces that are \emph{approximately equi-isoclinic} \cite{Lemmens.Seidel.IMP1973}, defined as ones with same principal angles between any two subspaces, regardless of the underlying mixing bases as long as the cumulative active subspace energy does not scale with the number of active subspaces.

\revise{We conclude our discussion of the fixed mixing bases model by reiterating that since this model is not invariant to the choice of bases, it does not address the subspace unmixing problem in its most general form. Nonetheless, as noted earlier, analysis of MSD under this model leads to equivalent results under the random directions model in a straightforward manner (cf.~Sec.~\ref{sec:geometry}). Further, the subspace unmixing problem in the context of group model selection and block-sparse compressed sensing is precisely given by the fixed mixing bases model. As such, the results reported in this section are also useful in their own right.}

\revise{\section{Performance of Marginal Subspace Detection Under the Random Directions Model}\label{sec:geometry}
While Sec.~\ref{sec:MSD_FMB} provides results for the subspace unmixing problem for the fixed mixing bases model, it does not provide us with the most general results for subspace unmixing. First, the results have been derived under the fixed mixing bases model, which is arguably not the best model for the problem of subspace unmixing. Second, the thresholds selected in Theorem~\ref{thm:FWER_MSD} require knowledge of the mixing bases due to their dependence on the average mixing coherences of the subspaces. Third, the performance of MSD described in Theorem~\ref{thm:NDP_MSD} is also a function of the average mixing coherences of the subspaces. A natural question to ask at this point is whether it is possible to derive results for subspace unmixing in the sense that they do not require explicit use of the mixing bases. It turns out that doing so is relatively easy as long as one considers the random directions model discussed in Sec.~\ref{sec:prob_form}.

In order to leverage the results of Sec.~\ref{sec:MSD_FMB} for the random directions model, we first use the probabilistic method to establish that any collection of subspaces $\cX_\Ns$ has associated with it at least one corresponding collection of orthonormal bases $\cU_\Ns := \big\{U_i: \tspan(U_i) = \cS_i, U_i^\tT U_i = I, i=1,\dots,\Ns\big\}$ such that $\rho_i(\cU_\Ns) = O\big(\frac{\gamma_{\textsf{rms},i}\sqrt{\log(\ds\Ns)}}{\sqrt{\Ns}}\big)$.
\begin{lemma}\label{lemma:mix_coh_ubound_RDmodel}
Let $d \geq 3$ and fix any $c_4 > 1$. Then every collection of subspaces $\cX_\Ns = \big\{\cS_i \in \fG(\ds,\Dh), i=1,\dots,\Ns\big\}$ has at least one collection of orthonormal bases $\cU_\Ns = \big\{U_i: \tspan(U_i) = \cS_i, U_i^\tT U_i = I, i=1,\dots,\Ns\big\}$ such that $$\rho_i = \frac{1}{\Ns-1}\Big\|\sum_{j\not=i} U_i^\tT U_j\Big\|_2 < \bar{\rho}_{i} := \frac{\gamma_{\emph{\textsf{rms}},i} \sqrt{\log(c_4 \ds^2\Ns)}}{\sqrt{c_0'(\Ns-1)}}, \quad i=1,\dots,\Ns.$$ Here, the parameter $c_0' := \frac{e^{-\frac{3}{2}}}{256} $ is an absolute positive constant.
\end{lemma}
The proof of this lemma is provided in Appendix~\ref{app:lemma3}. Lemma~\ref{lemma:mix_coh_ubound_RDmodel} helps us overcome all the challenges associated with the analysis of Sec.~\ref{sec:MSD_FMB} that have been outlined at the start of this section. Specifically, notice that all the results reported in Sec.~\ref{sec:MSD_FMB} under the fixed mixing bases model can have the $\rho_i$'s in them replaced with upper bounds on the average mixing coherences. To this end, Lemma~\ref{lemma:mix_coh_ubound_RDmodel} provides such upper bounds, $\bar{\rho}_{i}$, that only depend on the geometry of the underlying collection of subspaces. This, coupled with the fact that the MSD algorithm is invariant to the choice of subspace bases, implies that the results of Sec.~\ref{sec:MSD_FMB} immediately lead us to equivalent results for subspace unmixing that are fully characterized in terms of the local 2-subspace coherences and the quadratic-mean subspace coherences of the underlying subspaces. Nonetheless, there is still one point that has been left unaddressed in this discussion: \emph{it seems we are requiring the signal $x = \sum_{j=1}^\ns x_{i_j}$ to have been generated under the fixed mixing bases model, with the subspace bases being given by the ones in Lemma~\ref{lemma:mix_coh_ubound_RDmodel}}. We now argue that this requirement is in fact unnecessary for the case of $x$ being generated under the random directions model.

Let $x = \sum_{j=1}^\ns x_{i_j}$ be a signal generated according to the random directions model. We can then rewrite $x$ as
\begin{align}\label{eqn:RD.Model.FMB.Model}
    x = \sum_{j=1}^\ns  \cE_{i_j} \frac{x_{i_j}}{\|x_{i_j}\|_2} = \sum_{j=1}^\ns  U_{i_j} (\cE_{i_j} \ttheta_{i_j}) = \sum_{j=1}^\ns  U_{i_j} \theta_{i_j},
\end{align}
where $\theta_{i_j} := \cE_{i_j} \ttheta_{i_j}$, while the unit vector $\ttheta_{i_j} \in \R^\ds$ denotes the expansion of $x_{i_j}/\|x_{i_j}\|_2$ under the (fixed) collection of orthonormal bases $\cU_\Ns$ obtained in Lemma~\ref{lemma:mix_coh_ubound_RDmodel}; in other words, $\ttheta_{i_j} = U_{i_j}^\tT (x_{i_j}/\|x_{i_j}\|_2)$. Given that $\mathfrak{X}^n = \big(x_{i_1}/\|x_{i_1}\|_2,\dots,x_{i_n}/\|x_{i_n}\|_2\big)$ is drawn independently of $\As$, it follows that $\Xi^n := \big(\ttheta_{i_1},\dots,\ttheta_{i_n}\big)$ is also independent of $\As$ under the random directions model. Consequently, conditioning \eqref{eqn:RD.Model.FMB.Model} on $\Xi^n$ under the random directions model reduces it to the fixed mixing bases model. It is then straightforward to derive results equivalent to Lemma~\ref{lemma:right_tail_null_hypo} and Lemma~\ref{lemma:left_tail_alt_hypo} under the random directions model by combining the analysis of Sec.~\ref{sec:MSD_FMB} with Lemma~\ref{lemma:mix_coh_ubound_RDmodel} and noting that
\begin{align}
  \Pr\left(T_k(y) \gtreqqless \tau \big| \cH^k_0\right) = \int_{\Xi^n} \Pr\left(T_k(y) \gtreqqless \tau \big| \cH^k_0, \Xi^n \right) \lambda_{\mathfrak{B}^n}(\Xi^n).
\end{align}
This trivially leads to the following theorem concerning the \FWER~of MSD under the random directions model.
\begin{theorem}\label{thm:FWER_MSD_RD}
Fix any $\alpha \in [0,1]$ and define $$\bar{\rho}_k := \frac{\gamma_{\emph{\textsf{rms}},k} \sqrt{\log(c_4 \ds^2\Ns)}}{\sqrt{c_0'(\Ns-1)}}, \quad k=1,\dots,\Ns,$$ where $c_4 > 1$ is a fixed constant and $c_0'$ is as defined in Lemma~\ref{lemma:mix_coh_ubound_RDmodel}. Then the family-wise error rate of the marginal subspace detection (Algorithm~\ref{alg:MSD}) can be controlled at any level $\alpha \in [0,1]$ under the random directions model by selecting the decision thresholds $\{\tau_k\}_{k=1}^{\Ns}$ as follows:
\begin{enumerate}
\item In the case of bounded deterministic error $\eta$, select $$\tau_k = \left(\epsilon_\eta + \bar{\rho}_k \sqrt{\ns \cE_\As} + \frac{\gamma_{2,k}\Ns}{\Ns-\ns}\sqrt{c_0^{-1}\cE_\As\log\big(\tfrac{e^2 \Ns}{\alpha}\big)}\right)^2, \quad k=1,\dots,\Ns.$$

\item In the case of i.i.d. Gaussian noise $\eta$, select $$\tau_k = \left(\sigma\sqrt{\ds + 2\log\big(\tfrac{2\Ns}{\alpha}\big) + 2\sqrt{\ds\log\big(\tfrac{2\Ns}{\alpha}\big)}} + \bar{\rho}_k \sqrt{\ns \cE_\As} + \frac{\gamma_{2,k}\Ns}{\Ns-\ns}\sqrt{c_0^{-1}\cE_\As\log\big(\tfrac{e^2 2\Ns}{\alpha}\big)}\right)^2, \quad k=1,\dots,\Ns.$$
\end{enumerate}
\end{theorem}

Similar to Theorem~\ref{thm:FWER_MSD_RD}, one can also trivially derive an equivalent of Theorem~\ref{thm:NDP_MSD} under the random directions model by simply replacing $\rho_i$ with $\bar{\rho}_i := \frac{\gamma_{\emph{\textsf{rms}},i} \sqrt{\log(c_4 \ds^2\Ns)}}{\sqrt{c_0'(\Ns-1)}}$ in the definition of the set $\As_*$ within the theorem statement. In conclusion, the advantages of MSD outlined in Sec.~\ref{sec:MSD_FMB} for the fixed mixing bases model remain valid for the random directions model; the only difference here being that the measure of average mixing coherence gets replaced by the ratio of the measure of quadratic-mean subspace coherence and square-root of the total number of subspaces (modulo a logarithmic factor). Further, given that $\gamma_{\textsf{rms},i} = O(\gamma_{2,i})$, the subspace coherence condition \eqref{eqn:mixing_coh_cond} is simpler to satisfy under the random directions model for subspaces that are not too similar to each other (cf.~\eqref{eqn:local_2_coh_cond}). Finally, it is straightforward to combine this discussion with Theorem~\ref{thm:lin_scaling} and Theorem~\ref{thm:constant_energy} and conclude that the MSD algorithm also does not suffer from the square-root bottleneck under the random directions model.}

\section{Numerical Results}\label{sec:num_res}
In this section, we report results of numerical experiments that further shed light on the relationships between the local 2-subspace coherences, \revise{quadratic-mean subspace coherences}, average mixing coherences, and the MSD algorithm for the problem of subspace unmixing. The subspaces used in all these experiments are independently drawn at random from $\fG(\ds, \Dh)$ according to the natural uniform measure induced by the Haar measure on the \emph{Stiefel manifold} $\mathbb{S}(\ds,\Dh)$, which is defined as $\mathbb{S}(\ds,\Dh) := \{U \in \R^{\Dh \times \ds}: U^\tT U = I\}$. Computationally, we accomplish this by resorting to the numerical algorithm proposed in \cite{Mezzadri.NotA2007} for random drawing of elements from $\mathbb{S}(\ds,\Dh)$ according to the Haar measure. In doing so, we not only generate subspaces $\cX_\Ns = \{\cS_i\}_{i=1}^\Ns$ from $\fG(\ds, \Dh)$ \revise{for the random directions model}, but we also generate the associated mixing bases $\cB_\Ns = \{\Phi_i\}_{i=1}^\Ns$ from $\mathbb{S}(\ds,\Dh)$ \revise{for the fixed mixing bases model}. Mathematically, given a subspace $\cS_i \in \fG(\ds, \Dh)$ and its equivalence class in the Stiefel manifold $[\cS_i] \subset \mathbb{S}(\ds,\Dh)$, its associated mixing basis $\Phi_i \in \mathbb{S}(\ds,\Dh)$ is effectively drawn at random from $[\cS_i]$ according to the Haar measure on $[\cS_i]$. It is important to note here that once we generate the $\cS_i$'s and the $\Phi_i$'s, they remain fixed throughout our experiments. In other words, our results are not averaged over different realizations of the subspaces and the mixing bases; rather, they correspond to a \emph{fixed} set of subspaces \revise{(random directions models)} and mixing bases \revise{(fixed mixing bases model)}.

Our first set of experiments evaluates the local 2-subspace coherences \revise{and quadratic-mean subspace coherences} of the $\cS_i$'s and the average mixing coherences of the corresponding $\Phi_i$'s for different values of $\ds$, $\Dh$, and $\Ns$. The results of these experiments are reported in Figs.~\ref{fig:coh_errorbars} and \ref{fig:coh_hists}. Specifically, Fig.~\ref{fig:coh_errorbars}(a) and Fig.~\ref{fig:coh_errorbars}(b) plot $\sum_{i=1}^{\Ns} \gamma_{2,i}/\Ns$ as well as the range of the $\gamma_{2,i}$'s using error bars for $\Ns = 1500$ and $\Ns = 2000$, respectively. Similarly, Fig.~\ref{fig:coh_errorbars}(c) and Fig.~\ref{fig:coh_errorbars}(d) plot \revise{plot $\sum_{i=1}^{\Ns} \gamma_{\textsf{rms},i}/\Ns$ as well as the range of the $\gamma_{\textsf{rms},i}$'s using error bars for $\Ns = 1500$ and $\Ns = 2000$, respectively. Finally, Fig.~\ref{fig:coh_errorbars}(e) and Fig.~\ref{fig:coh_errorbars}(f) plot} $\sum_{i=1}^{\Ns} \rho_i/\Ns$ as well as the range of the $\rho_i$'s using error bars for $\Ns = 1500$ and $\Ns = 2000$, respectively. It can be seen from these figures that \revise{all three coherence measures under consideration} decrease with an increase in $\Dh$, while they increase with an increase in $\ds$. In addition, it appears from these figures that the $\gamma_{2,i}$'s and the $\rho_i$'s start concentrating around their average values for larger values of \revise{$\ds$ and} $\Dh$. \revise{In contrast, the $\gamma_{\textsf{rms},i}$'s appear highly concentrated around their average values, which is attributable to the random generation of the $\cS_i$'s.} Another important thing to notice from Fig.~\ref{fig:coh_errorbars} is that the average mixing coherences tend to be more than two orders of magnitude smaller than the local 2-subspace coherences, which is indeed desired \revise{under the fixed mixing bases model} according to the discussion in Sec.~\ref{sec:MSD_FMB}. \revise{We can also make a similar observation from Fig.~\ref{fig:coh_errorbars}(c)--(d) about the $\bar{\rho}_i$'s defined in Theorem~\ref{thm:FWER_MSD_RD} for the random directions model; e.g., $\bar{\rho}_i = (7\times 10^{-2}) \gamma_{\textsf{rms},i}$ for $\ds = 3$ and $\Ns = 2000$ under the assumption of $c_0' = c_4 = 1$ (more on this assumption later).} Finally, since the error bars in Fig.~\ref{fig:coh_errorbars} do not give insights into distributions of the $\gamma_{2,i}$'s, \revise{$\gamma_{\textsf{rms},i}$'s} and $\rho_i$'s, we also plot histograms of the \revise{three} coherences in Fig.~\ref{fig:coh_hists} for $\Ns=2000$ corresponding to $\Dh=600$ (Figs.~\ref{fig:coh_hists}(a), \ref{fig:coh_hists}(c)\revise{, and \ref{fig:coh_hists}(e)}) and $\Dh=1400$ (Figs.~\ref{fig:coh_hists}(b), \ref{fig:coh_hists}(d)\revise{, and \ref{fig:coh_hists}(f)}).

\begin{figure}[p]
\centering
\begin{tabular}{ccc}
{\footnotesize (a)} & \hfill & {\footnotesize (b)}\\
\includegraphics[width=3in]{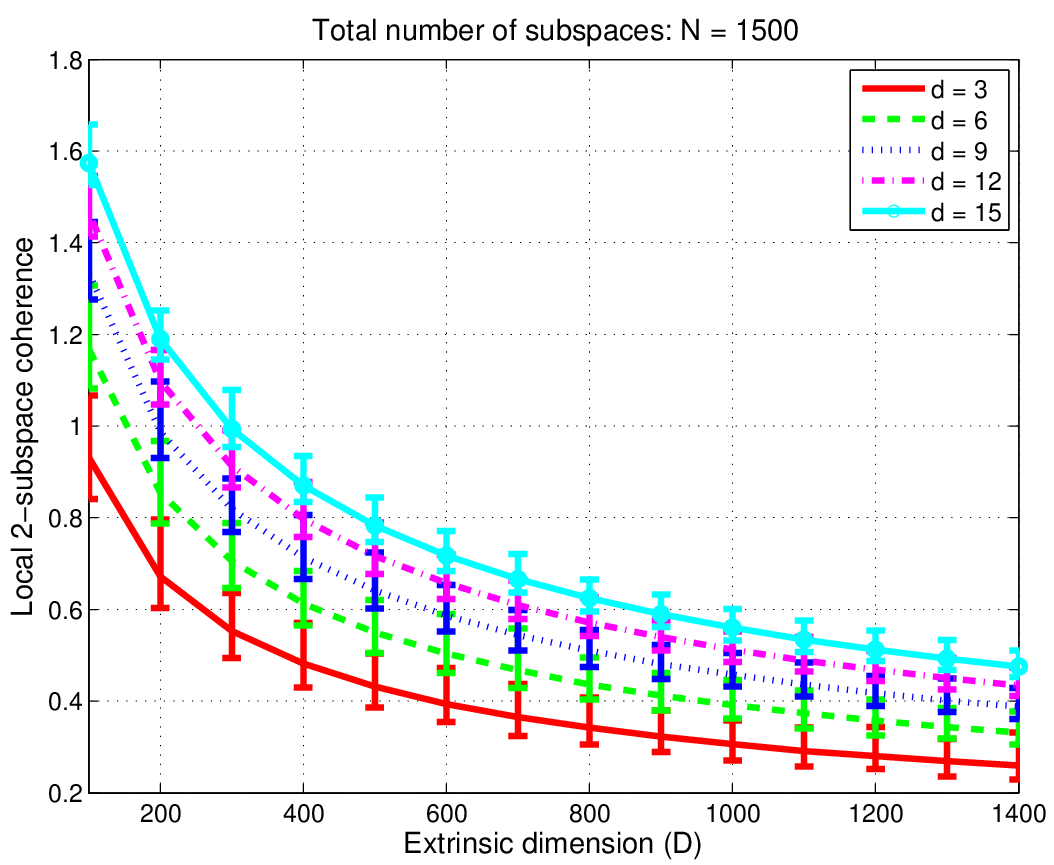}& \hfill & \includegraphics[width=3in]{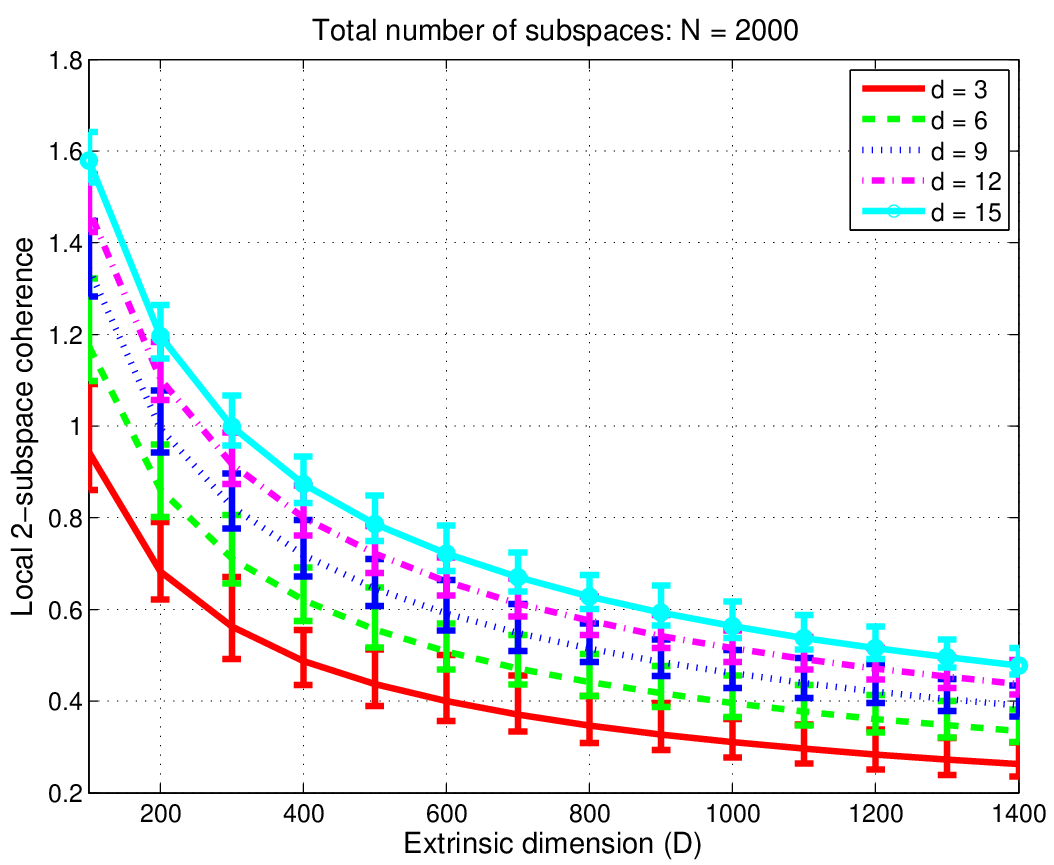}\\
{\footnotesize (c)} & \hfill & {\footnotesize (d)}\\
\includegraphics[width=3in]{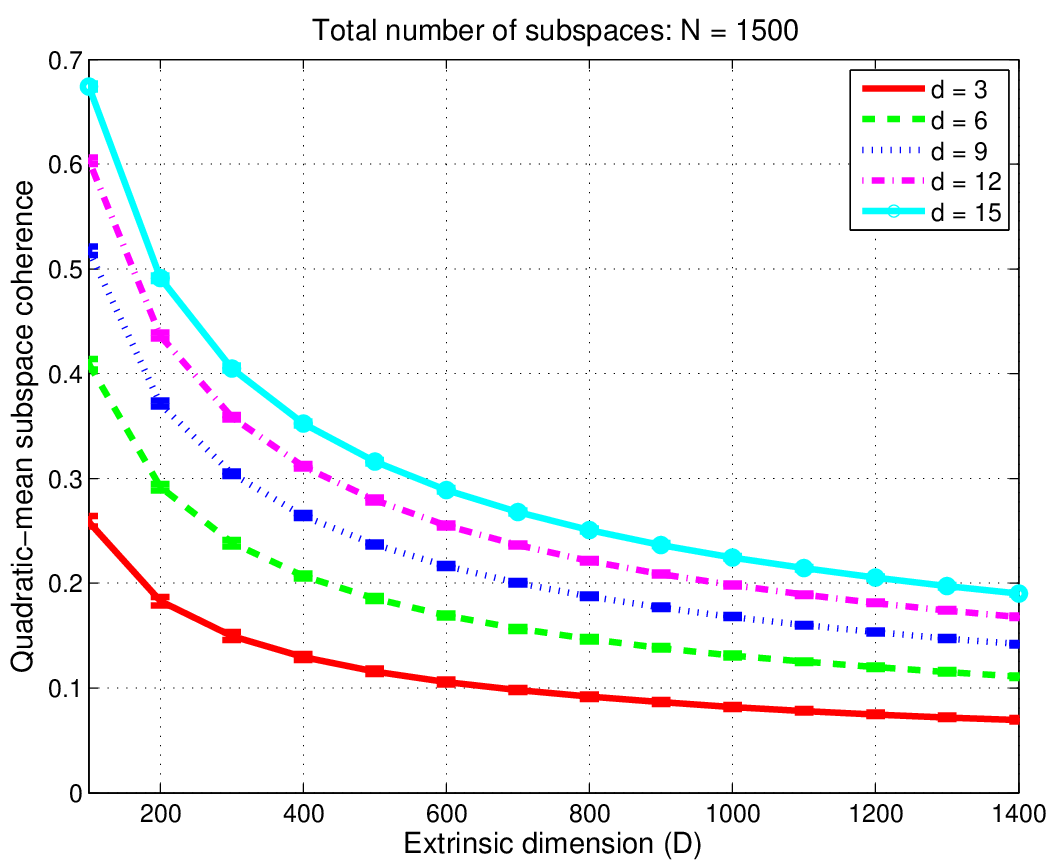}& \hfill & \includegraphics[width=3in]{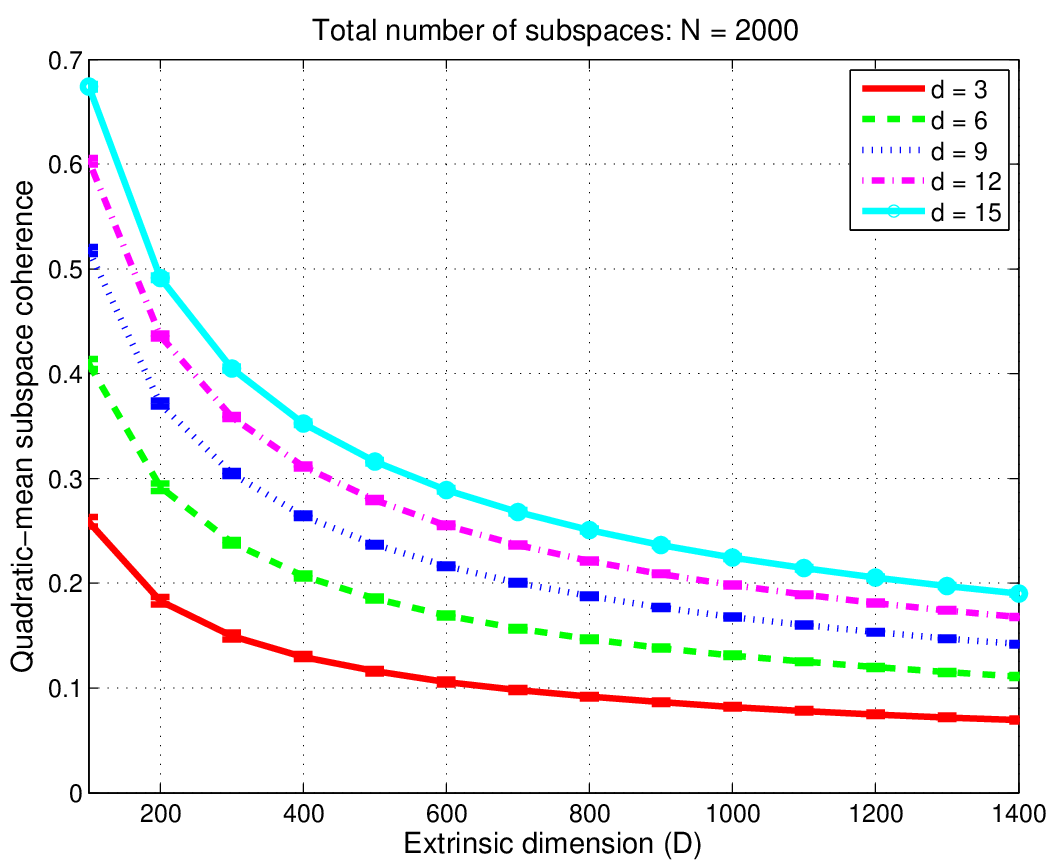}\\
{\footnotesize (e)} & \hfill & {\footnotesize (f)}\\
\includegraphics[width=3in]{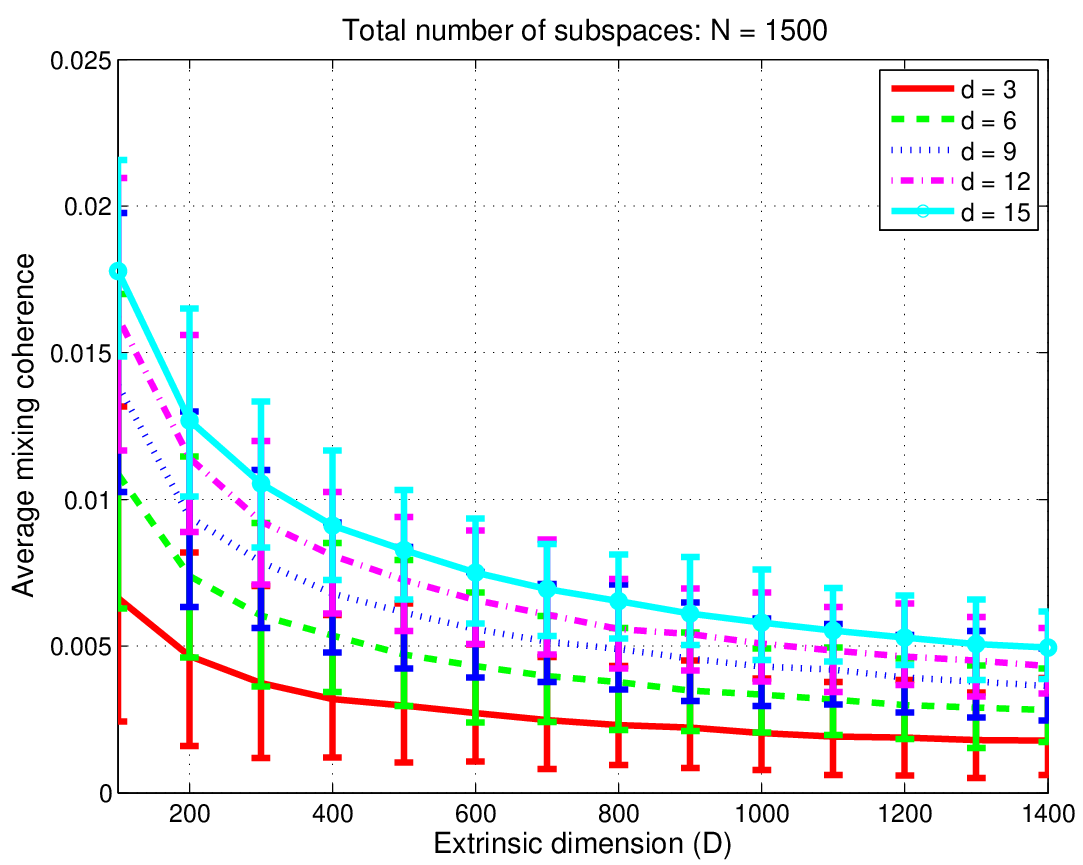}& \hfill & \includegraphics[width=3in]{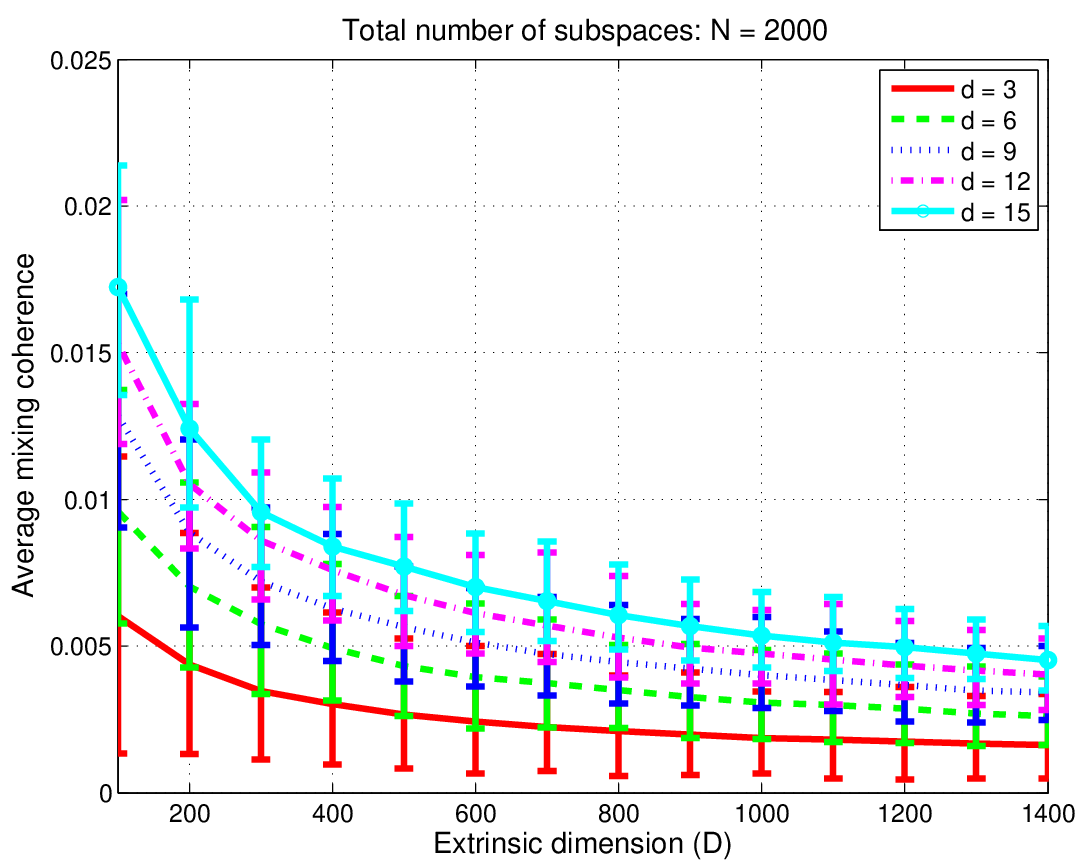}
\end{tabular}
\caption{\revise{Plots of local 2-subspace coherences, quadratic-mean subspace coherences, and average mixing coherences for different values of $\ds$, $\Dh$, and $\Ns$. (a) and (b) correspond to local 2-subspace coherences, (c) and (d) correspond to quadratic-mean subspace coherences, and (e) and (f) correspond to average mixing coherences. The error bars in the plots depict the range of coherences for the different subspaces.}}
\label{fig:coh_errorbars}
\end{figure}

\begin{figure}[p]
\centering
\begin{tabular}{ccc}
{\footnotesize (a) $\Dh=600$} & \hfill & {\footnotesize (b) $\Dh=1400$}\\
\includegraphics[width=3in]{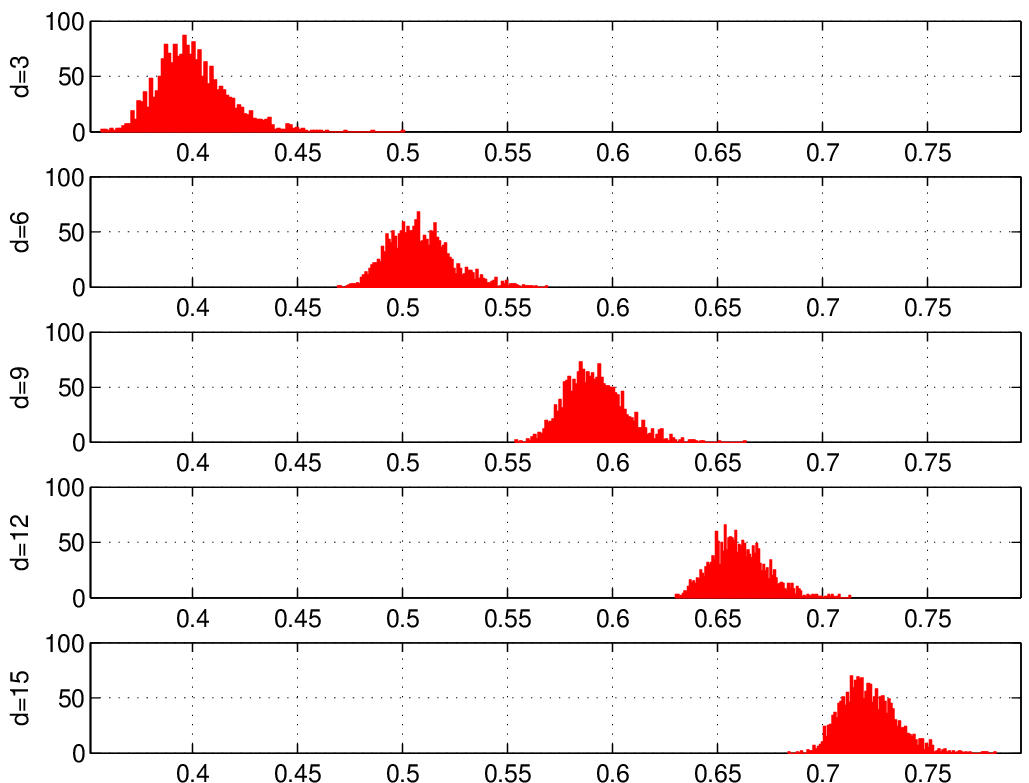}& \hfill & \includegraphics[width=3in]{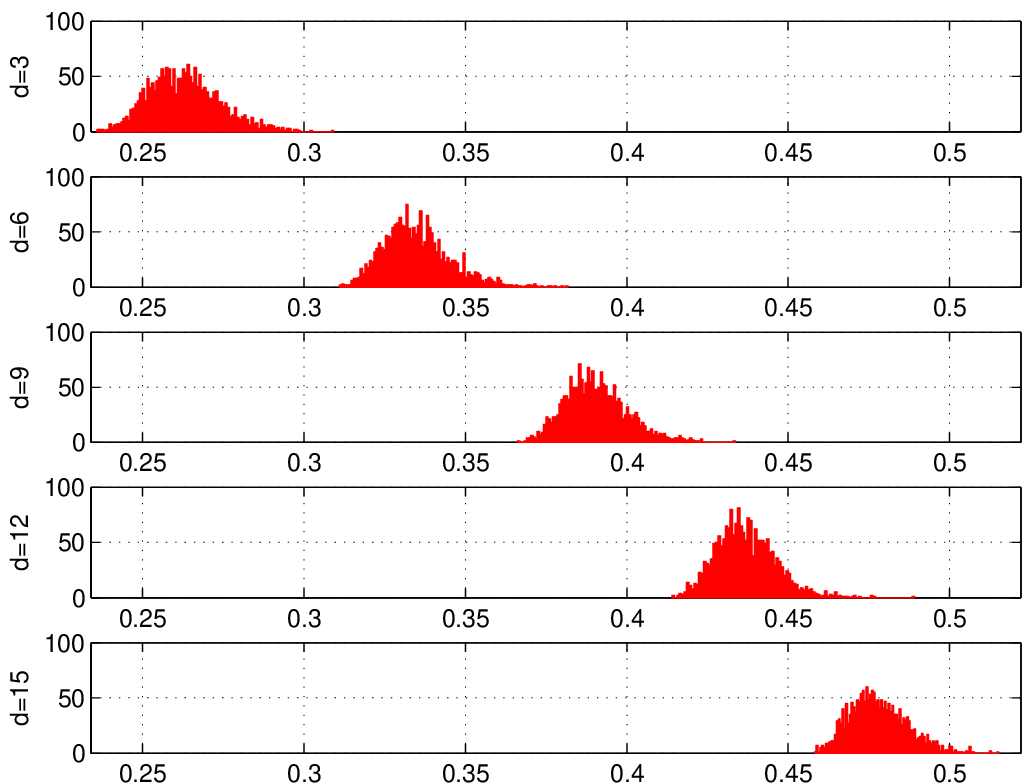}\\
{\footnotesize (c) $\Dh=600$} & \hfill & {\footnotesize (d) $\Dh=1400$}\\
\includegraphics[width=3in]{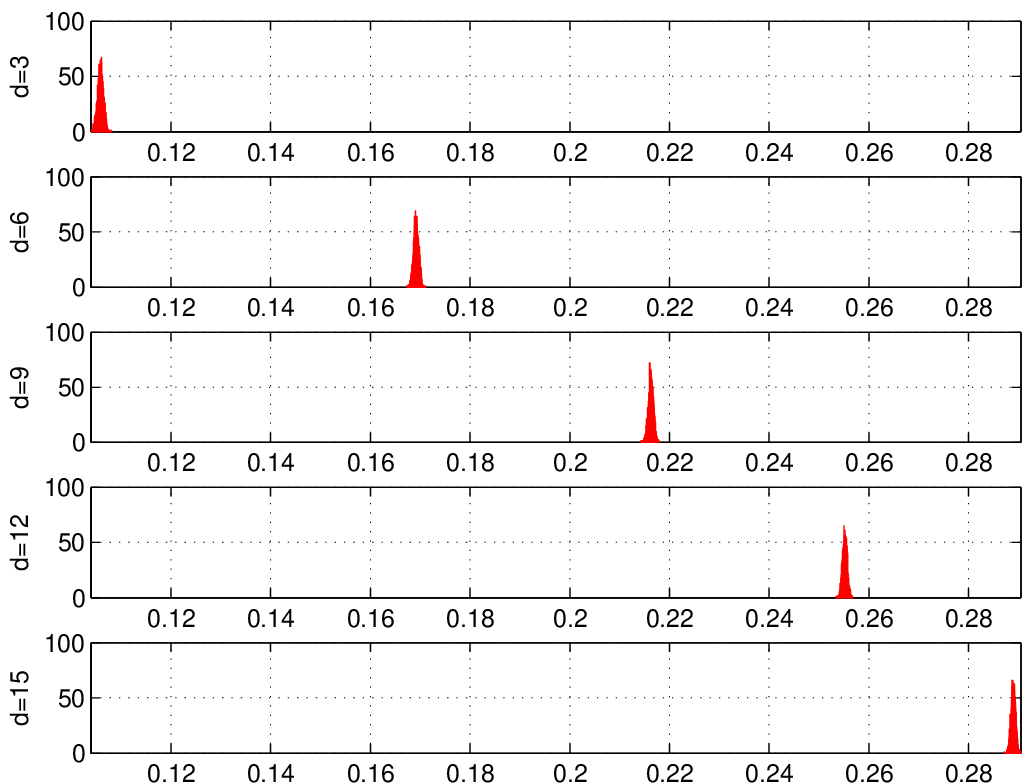}& \hfill & \includegraphics[width=3in]{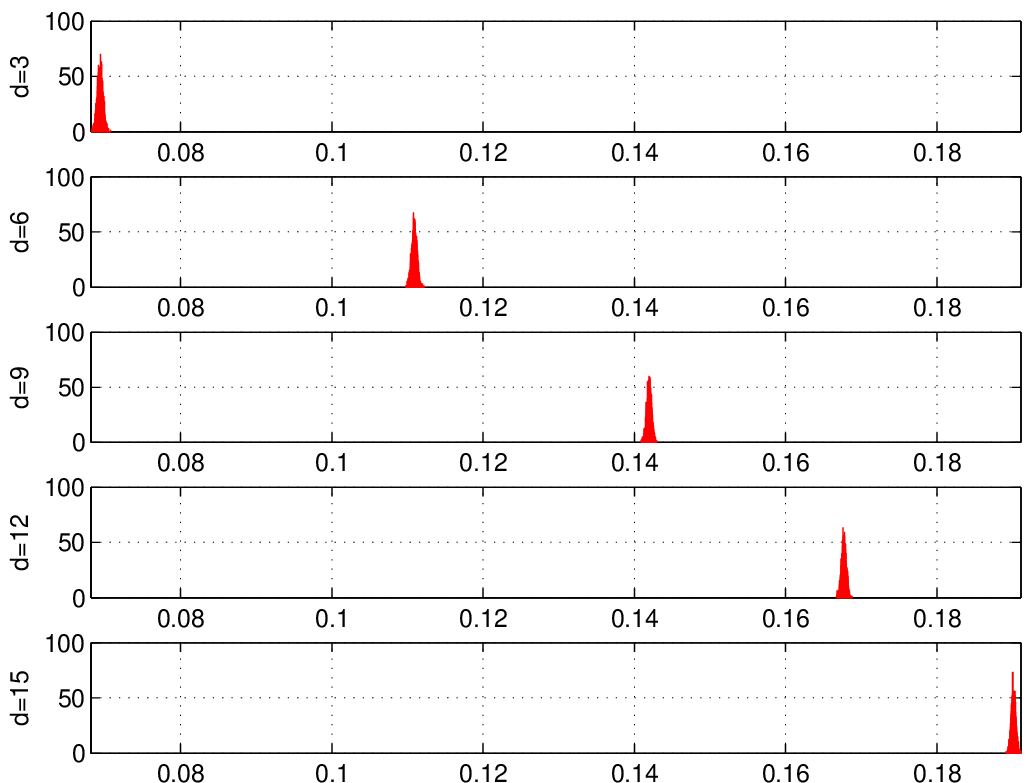}\\
{\footnotesize (e) $\Dh=600$} & \hfill & {\footnotesize (f) $\Dh=1400$}\\
\includegraphics[width=3in]{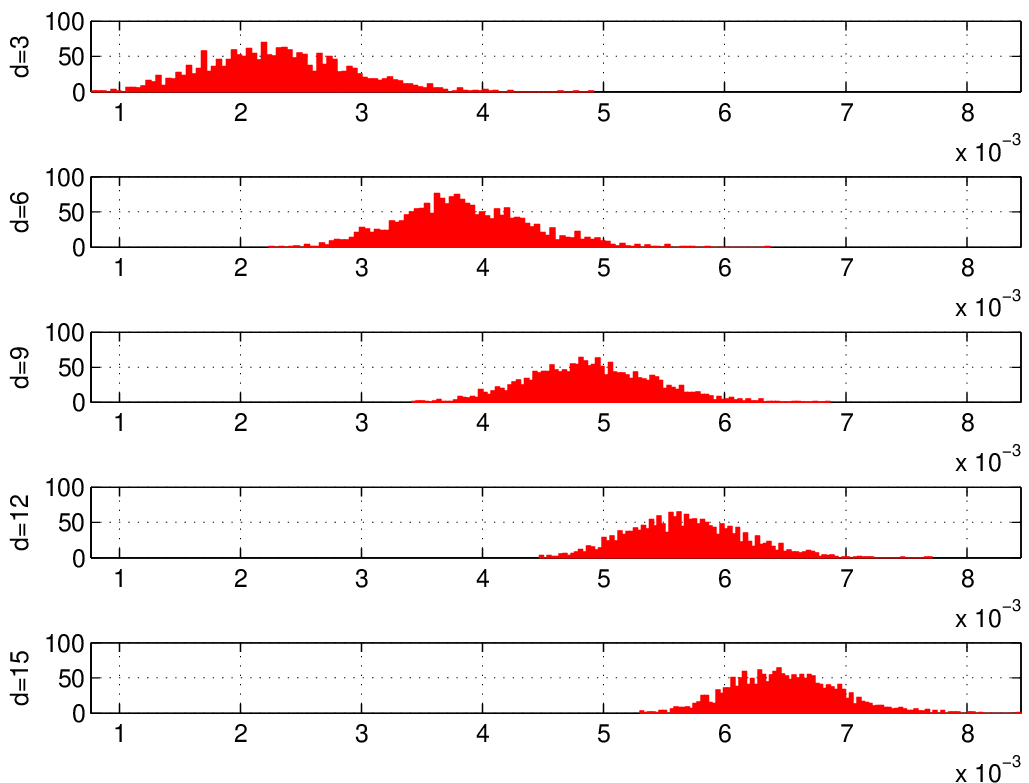}& \hfill & \includegraphics[width=3in]{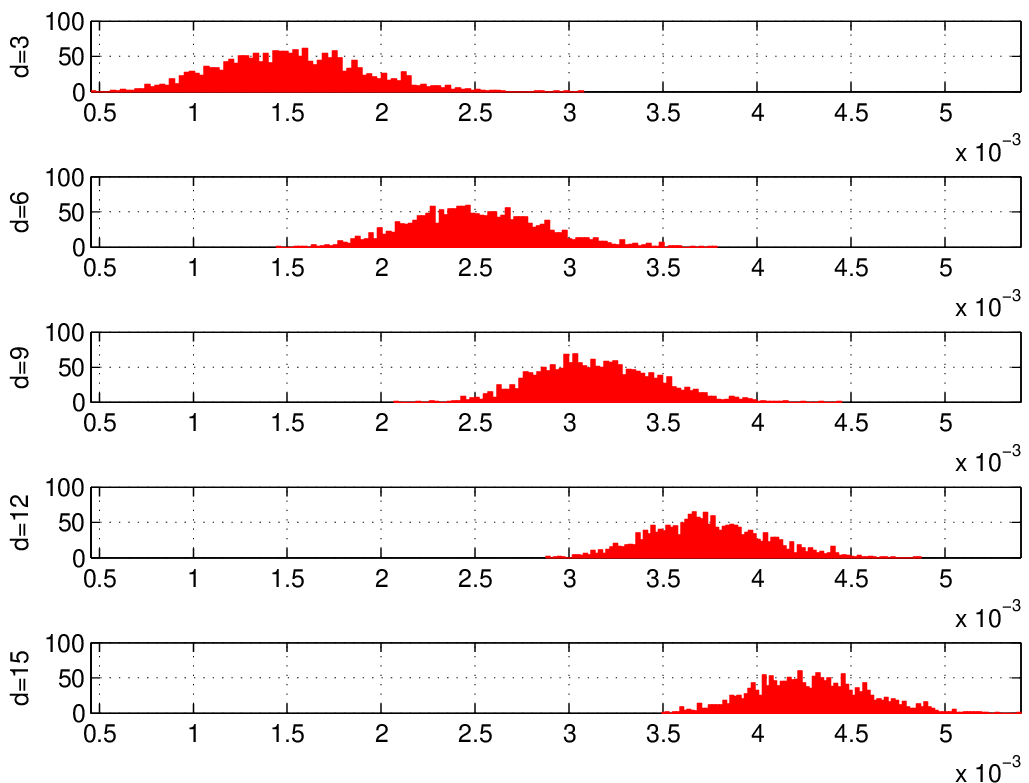}
\end{tabular}
\caption{\revise{Histograms of local 2-subspace coherences, quadratic-mean subspace coherences, and average mixing coherences for $\Ns=2000$ and different values of $\ds$. (a) and (b) correspond to local 2-subspace coherences, (c) and (d) correspond to quadratic-mean subspace coherences, and (e) and (f) correspond to average mixing coherences.}}
\label{fig:coh_hists}
\end{figure}

Our second set of experiments evaluates the performance of the MSD algorithm for subspace unmixing \revise{under both the fixed mixing bases and the random directions models}. We run these experiments for \emph{fixed} subspaces and mixing bases for the following four sets of choices for $(\ds,\Dh,\Ns)$: $(3, 600, 2000), (3, 1400, 2000), (15, 600, 2000)$, and $(15, 1400, 2000)$. The results reported for these experiments are averaged over 5000 realizations of subspace activity patterns, mixing coefficients, and additive Gaussian noise. In all these experiments, we use $\sigma = 0.01$ and $\cE_\As = \ns$, divided equally among all active subspaces, which means that all active subspaces lie above the additive noise floor. In terms of the selection of thresholds for Algorithm~\ref{alg:MSD}, we rely on Theorem~\ref{thm:FWER_MSD} \revise{and Theorem~\ref{thm:FWER_MSD_RD} for fixed mixing bases model and random directions model, respectively, but} with a small caveat. Since our proofs use a number of probabilistic bounds, the theorem statements invariably result in conservative thresholds. In order to remedy this, we use the thresholds $\tilde{\tau}_k := c_1^2 \tau_k$ with $\tau_k$ as in Theorem~\ref{thm:FWER_MSD} \revise{and Theorem~\ref{thm:FWER_MSD_RD} \emph{but} using $c_0 = c_4 = 1$, $c_1 \in (0,1)$, and $c_0' \in [1, \ds]$}. We tune the parameter $c_1$ using cross validation and set $c_1 = 0.136$ and $c_1 = 0.107$ for $\ds=3$ and $\ds=15$, respectively. \revise{In order to understand the effects of tuning $c_0'$ for the random directions models, we set $c_0' = c_0$ for $\ds = 3$, while we explicitly tune it by cross validation for $\ds = 15$ and set $c_0' = 15$ in this case.} Finally, we set the final thresholds to control the \FWER~in all these experiments at level $\alpha = 0.1$.

\begin{figure}[ht]
\centering
\begin{tabular}{ccc}
{\footnotesize (a)} & \hfill & {\footnotesize (b)}\\
\includegraphics[width=3in]{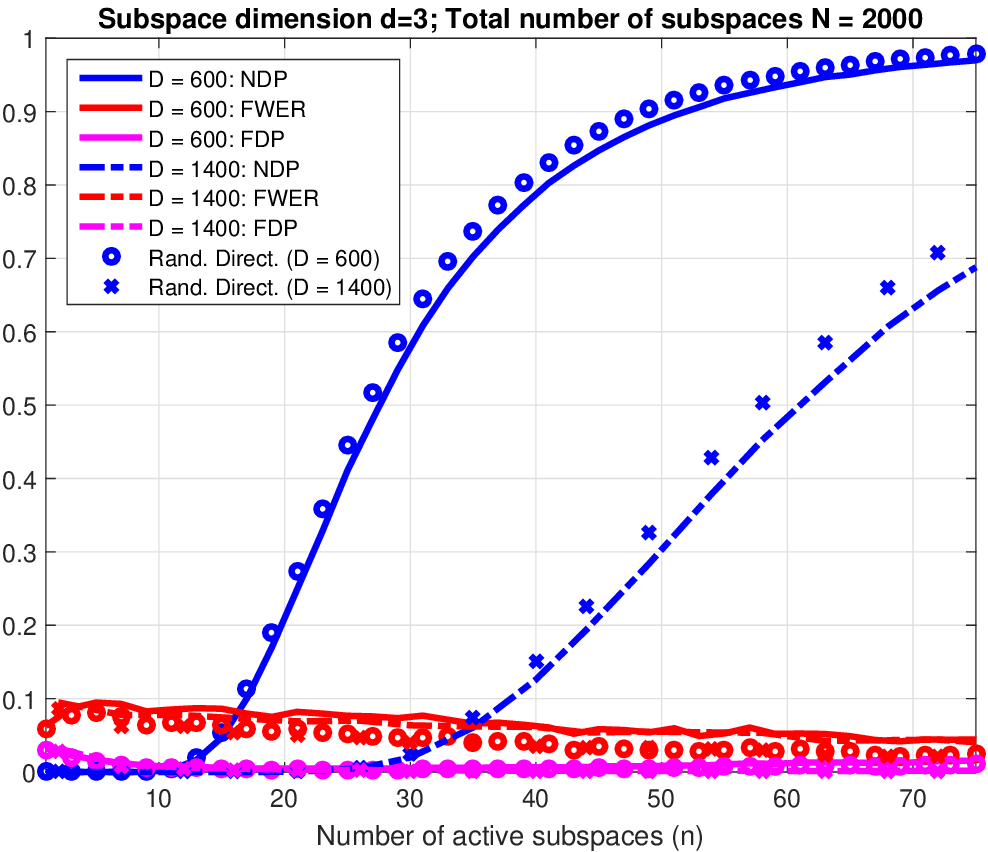}& \hfill & \includegraphics[width=3in]{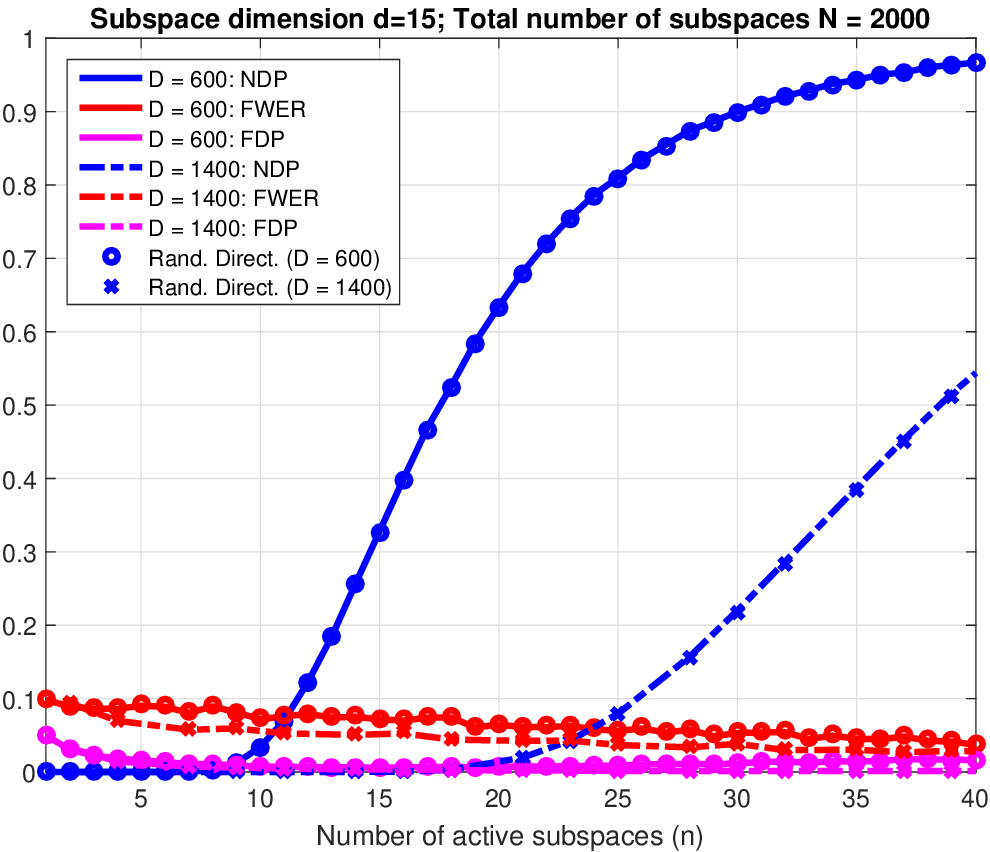}
\end{tabular}
\caption{Plots of \FWER, \NDP, and \FDP~as a function of the number of active subspaces, $\ns$\revise{, under both the fixed mixing bases (FMB) and the random directions (RD) models. These plots correspond to $\Dh=600$ (FMB: solid lines; RD: circles ($\boldsymbol{\circ}$)) and $\Dh=1400$ (FMB: dashed lines; RD: crosses ($\boldsymbol{\times}$)).}}
\label{fig:MSD_perf}
\end{figure}

The results of these experiments for our choices of the parameters are reported in Fig.~\ref{fig:MSD_perf}(a) and Fig.~\ref{fig:MSD_perf}(b) for $\ds=3$ and $\ds=15$, respectively. We not only plot the \FWER~and the \NDP~in these figures \revise{for both the fixed mixing bases and the random directions models}, but we also plot another metric of \emph{false-discovery proportion} (\FDP), defined as $\FDP := \frac{|\whAs \setminus \As|}{|\whAs|}$, as a measure of the \FDR. Indeed, the expectation of the \FDP~is the \FDR~\cite{Farcomeni.SMiMR2008}. \revise{We first compare the \FWER~plots for $\Dh=600$ and $\Dh=1400$ in these figures for the fixed mixing bases model (solid and dashed lines)}. We can see from Fig.~\ref{fig:coh_hists} that the $\gamma_{2,i}$'s and the $\rho_i$'s are smaller for $\Dh=1400$, which means that the thresholds $\tilde{\tau}_k$'s are also smaller for $\Dh=1400$ (cf.~Theorem~\ref{thm:FWER_MSD}). But Fig.~\ref{fig:MSD_perf} shows that the \FWER~for $\Dh=1400$ mostly remains below $\Dh=600$, which suggests that Theorem~\ref{thm:FWER_MSD} is indeed capturing the correct relationship between the \FWER~of MSD and the properties of the underlying \revise{mixing bases}. In addition, the \NDP~plots \revise{(solid and dashed lines)} in these figures for $\Dh=600$ and $\Dh=1400$ \revise{under the fixed mixing bases model} also help validate Theorem~\ref{thm:NDP_MSD}. Specifically, Theorem~\ref{thm:NDP_MSD} suggests that the \NDP~of MSD should remain small for larger values of $\ns$ as long as the $\gamma_{2,i}$'s and the $\rho_i$'s remain small. Stated differently, since the $\gamma_{2,i}$'s and the $\rho_i$'s are smaller for $\Dh=1400$ than for $\Dh=600$ (cf.~Fig~\ref{fig:coh_hists}), Theorem~\ref{thm:NDP_MSD} translates into a smaller \NDP~figure for larger values of $\ns$ for $\Dh=1400$. It can be seen from the \NDP~plots in Fig.~\ref{fig:MSD_perf} that this is indeed the case. \revise{Finally, we turn our attention to \FWER,~\NDP, and \FDP~plots in Fig.~\ref{fig:MSD_perf}(a) and Fig.~\ref{fig:MSD_perf}(b) for thresholds under the random directions model (circles ($\boldsymbol{\circ}$) and crosses ($\boldsymbol{\times}$)). Careful examination of these plots confirm that indeed: ($i$) the MSD algorithm does not require explicit knowledge of the mixing bases for calculations of the decision thresholds; and ($ii$) the upper bounds derived in Lemma~\ref{lemma:mix_coh_ubound_RDmodel} for the $\rho_i$'s are (order-wise) tight. Specifically, it can be seen from Fig.~\ref{fig:MSD_perf} that the thresholds derived in Sec.~\ref{sec:geometry} for the random directions model result in performance that is either close to ($\ds=3$ with untuned $c_0'$) or similar to ($\ds=15$ with tuned $c_0'$) the one using thresholds that rely on knowledge of the mixing bases.}

\section{Conclusion}\label{sec:conc}
In this paper, we motivated and posed the problem of subspace unmixing \revise{under the parsimonious subspace-sum (PS3) model} as well as discussed its connections with problems in wireless communications, hyperspectral imaging, high-dimensional statistics and compressed sensing. We proposed and analyzed a low-complexity algorithm, termed \emph{marginal subspace detection} (MSD), that solves the subspace unmixing problem under the \revise{PS3} model by turning it into a multiple hypothesis testing problem. We showed that the MSD algorithm can be used to control the family-wise error rate at any level $\alpha \in [0,1]$ for an arbitrary collection of subspaces on the Grassmann manifold \revise{under two random signal generation models}. We also established that the MSD algorithm allows for linear scaling of the number of active subspaces as a function of the ambient dimension. Numerical results presented in the paper further validated the usefulness of the MSD algorithm and the accompanying analysis. Future work in this direction includes design and analysis of algorithms that perform better than the MSD algorithm as well as study of the subspace unmixing problem under mixing models other than the \revise{PS3} model.

\begin{appendices}
\section{\revise{Proof of Lemma~\ref{lemma:right_tail_null_hypo}}}\label{app:lemma1}
We begin by defining $\wtT_k(y) := \sqrt{T_k(y)}$ and noting $\wtT_k(y) \leq \big\|\sum_{j=1}^\ns \Phi_k^\tT\Phi_{i_j} \theta_j\big\|_2 + \big\|\Phi_k^\tT\eta\big\|_2$. In order to characterize the right-tail probability of $T_k(y)$ under $\cH_0^k$, it suffices to characterize the right-tail probabilities of $Z_1^k := \big\|\sum_{j=1}^\ns \Phi_k^\tT\Phi_{i_j} \theta_j\big\|_2$ and $Z_2^k := \big\|\Phi_k^\tT\eta\big\|_2$ under $\cH_0^k$. This is rather straightforward in the case of $Z_2^k$. In the case of deterministic error $\eta$, we have $Z_2^k \geq \epsilon_\eta$ with zero probability. In the case of $\eta$ being distributed as $\cN(0, \sigma^2 I)$, we have that $\eta_k := \Phi_k^\tT\eta \in \R^\ds \sim \cN(0, \sigma^2 I)$. In that case, the right-tail probability of $Z_2^k$ can be obtained by relying on a concentration of measure result in \cite[Sec.~4,~Lem.~1]{Laurent.Massart.AS2000} for the sum of squares of i.i.d. Gaussian random variables. Specifically, it follows from \cite{Laurent.Massart.AS2000} that $\forall \delta_2 > 0$,
\begin{align}
\label{eqn:laurent_massart_bound}
\Pr\left(Z_2^k \geq \sigma\sqrt{\ds + 2\delta_2 + 2\sqrt{\ds\delta_2}}\right) \leq \exp(-\delta_2).
\end{align}

We now focus on the right-tail probability of $Z_1^k$, conditioned on the null hypothesis. Recall that $\As$ is a random $\ns$-subset of $\{1,2,\dots,\Ns\}$ with $\Pr(\As = \{i_1,i_2,\dots,i_\ns\}) = 1/\binom{\Ns}{\ns}$. Therefore, defining $\bar{\Pi} := \left(\pi_1,\dots,\pi_\Ns\right)$ to be a random permutation of $\{1,\dots,\Ns\}$ and using $\Pi := \left(\pi_1,\dots,\pi_\ns\right)$ to denote the first $\ns$-elements of $\bar{\Pi}$, the following equality holds in distribution:
\begin{align}
\label{eqn:eq_in_distrib_1}
    \Big\|\sum_{j=1}^\ns \Phi_k^\tT\Phi_{i_j} \theta_j \Big\|_2 \ : \ k \not\in \As \ \overset{dist}{=} \
    \Big\|\sum_{j=1}^\ns \Phi_k^\tT\Phi_{\pi_j} \theta_j\Big\|_2 \ : \ k \not\in \Pi.
\end{align}
We now define a probability event $E^k_0 := \big\{\Pi = \left(\pi_1,\dots,\pi_\ns\right) : k \not\in \Pi\big\}$ and notice from \eqref{eqn:eq_in_distrib_1} that
\begin{align}
\label{eqn:prob_bd_H0}
    \Pr(Z_1^k \geq \delta_1 \big| \cH_0^k) =
        \Pr\Bigg(\Big\|\sum_{j=1}^\ns \Phi_k^\tT\Phi_{\pi_j} \theta_j\Big\|_2 \geq
        \delta_1 \big| E^k_0\Bigg).
\end{align}
The rest of this proof relies heavily on a Banach-space-valued Azuma's inequality (Proposition~\ref{prop:azumaineq}) stated in Appendix~\ref{app:azuma}. In order to make use of Proposition~\ref{prop:azumaineq}, we construct an $\R^\ds$-valued Doob's martingale $\left(M_0,M_1,\dots,M_\ns\right)$ on $\sum_{j=1}^\ns \Phi_k^\tT\Phi_{\pi_j} \theta_j$ as follows:
\begin{align}
    M_0 &:= \sum_{j=1}^\ns \Phi_k^\tT \E\big[\Phi_{\pi_j}\big|E_0^k\big] \theta_j, \quad \text{and}\\
    M_{\ell} &:= \sum_{j=1}^\ns \Phi_k^\tT \E\big[\Phi_{\pi_j}\big|\pi_1^\ell, E_0^k\big] \theta_j, \ \ell=1,\dots,\ns,
\end{align}
where $\pi_1^\ell := (\pi_1,\dots,\pi_\ell)$ denotes the first $\ell$ elements of $\Pi$. The next step involves showing that the constructed martingale has bounded $\ell_2$ differences. In order for this, we define
\begin{align}
    M_\ell(u) := \sum_{j=1}^\ns \Phi_k^\tT \E\big[\Phi_{\pi_j}\big|\pi_1^{\ell-1}, \pi_\ell = u, E_0^k\big] \theta_j
\end{align}
for $u \in \{1,\dots,\Ns\} \setminus \{k\}$ and $\ell=1,\dots,\ns$. It can then be established using techniques very similar to the ones used in the \emph{method of bounded differences} for scalar-valued martingales that \cite{McDiarmid.SiC1989,Motwani.Raghavan.Book1995}
\begin{align}
\label{eqn:martingale_bdd_1}
    \|M_\ell - M_{\ell-1}\|_2 \leq \sup_{u,v} \|M_\ell(u) - M_\ell(v)\|_2.
\end{align}

In order to upper bound $\|M_\ell(u) - M_\ell(v)\|_2$, we define a $\Dh \times \ds$ matrix $\tPhi_{\ell,j}^{u,v}$ as
\begin{align}
    \tPhi_{\ell,j}^{u,v} := \E\big[\Phi_{\pi_j}\big|\pi_1^{\ell-1}, \pi_\ell = u, E_0^k\big] -
        \E\big[\Phi_{\pi_j}\big|\pi_1^{\ell-1}, \pi_\ell = v, E_0^k\big], \quad \ell=1,\dots,\ns,
\end{align}
and note that $\tPhi_{\ell,j}^{u,v} = 0$ for $j < \ell$ and $\tPhi_{\ell,j}^{u,v} = \Phi_u - \Phi_v$ for $j = \ell$. In addition, notice that the random variable $\pi_j$ conditioned on $\big\{\pi_1^{\ell-1}, \pi_\ell = u, E_0^k\big\}$ has a uniform distribution over $\{1,\dots,\Ns\} \setminus \{\pi_1^{\ell-1}, u, k\}$, while $\pi_j$ conditioned on $\big\{\pi_1^{\ell-1}, \pi_\ell = v, E_0^k\big\}$ has a uniform distribution over $\{1,\dots,\Ns\} \setminus \{\pi_1^{\ell-1}, v, k\}$. Therefore, we get $\forall j > \ell$,
\begin{align}
    \tPhi_{\ell,j}^{u,v} = \frac{1}{\Ns-\ell-1}\left(\Phi_u - \Phi_v\right).
\end{align}
It now follows from the preceding discussion that
\begin{align}
\nonumber
    \|M_\ell(u) - M_\ell(v)\|_2 = \big\|\sum_{j=1}^\ns \Phi_k^\tT \tPhi_{\ell,j}^{u,v} \theta_j\big\|_2 &\stackrel{(a)}{\leq} \sum_{j=1}^\ns \big\|\Phi_k^\tT \tPhi_{\ell,j}^{u,v}\big\|_2 \|\theta_j\|_2\\
\nonumber
        &\leq \big\|\Phi_k^\tT\left(\Phi_u - \Phi_v\right)\big\|_2 \|\theta_\ell\|_2 + \frac{\sum_{j > \ell} \big\|\Phi_k^\tT\left(\Phi_u - \Phi_v\right)\big\|_2 \|\theta_j\|_2}{\Ns-\ell-1}\\
\label{eqn:martingale_bdd_2}
        &\leq \big(\gamma(\cS_k, \cS_u) + \gamma(\cS_k, \cS_v)\big)\left(\|\theta_\ell\|_2 + \frac{\sum_{j > \ell} \|\theta_j\|_2}{\Ns-\ell-1}\right),
\end{align}
where $(a)$ is due to the triangle inequality and submultiplicativity of the operator norm. It then follows from \eqref{eqn:martingale_bdd_1}, \eqref{eqn:martingale_bdd_2} and definition of the local $2$-subspace coherence that
\begin{align}
\label{eqn:martingale_bdd_3}
    \|M_\ell - M_{\ell-1}\|_2 \leq
        \underbrace{\gamma_{2,k}\left(\|\theta_\ell\|_2 +
        \frac{\sum_{j > \ell} \|\theta_j\|_2}{\Ns-\ell-1}\right)}_{b_\ell}.
\end{align}

The final bound we need in order to utilize Proposition~\ref{prop:azumaineq} is that on $\|M_0\|_2$. To this end, note that $\pi_j$ conditioned on $E_0^k$ has a uniform distribution over $\{1,\dots,\Ns\} \setminus \{k\}$. It therefore follows that
\begin{align}
\label{eqn:bound_M0_H0}
    \|M_0\|_2 = \Big\|\sum_{j=1}^\ns \Phi_k^\tT \big(\sum_{\substack{q=1\\q\not=k}}^{\Ns}\frac{\Phi_q}{\Ns - 1}\big) \theta_j\Big\|_2 \stackrel{(b)}{\leq}
                        \frac{1}{\Ns-1}\Big\|\sum_{\substack{q=1\\q\not=k}}^{\Ns} \Phi_k^\tT \Phi_q\Big\|_2 \Big\|\sum_{j=1}^\ns\theta_j\Big\|_2
                        \stackrel{(c)}{\leq} \rho_k \sqrt{\ns \cE_\As}.
\end{align}
Here, $(b)$ is again due to submultiplicativity of the operator norm, while $(c)$ is due to definitions of the average mixing coherence and the cumulative active subspace energy as well as the triangle inequality and the Cauchy--Schwarz inequality. Next, we make use of \cite[Lemma~B.1]{Donahue.etal.CA1997} to note that $\zeta_{\cB}(\tau)$ defined in Proposition~\ref{prop:azumaineq} satisfies $\zeta_{\cB}(\tau) \leq \tau^2/2$ for $(\cB, \|\cdot\|) \equiv \big(L_2(\R^\ds), \|\cdot\|_2\big)$. Consequently, under the assumption $\delta_1 > \rho_k \sqrt{\ns \cE_\As}$, it can be seen from our construction of the Doob martingale $\left(M_0,M_1,\dots,M_\ns\right)$ that
\begin{align}
\nonumber
    \Pr\Bigg(\Big\|\sum_{j=1}^\ns \Phi_k^\tT\Phi_{\pi_j} \theta_j\Big\|_2 \geq
        \delta_1 \big| E^k_0\Bigg) &= \Pr\big(\|M_\ns\|_2 \geq
        \delta_1 \big| E^k_0\big) = \Pr\big(\|M_\ns\|_2 - \|M_0\|_2 \geq
        \delta_1 - \|M_0\|_2\big| E^k_0\big)\\
\nonumber
    &\stackrel{(d)}{\leq} \Pr\left(\|M_\ns - M_0\|_2 \geq
        \delta_1 - \rho_k \sqrt{\ns \cE_\As} \,\big| E^k_0\right)\\
\label{eqn:azuma_1}
    &\stackrel{(e)}{\leq} e^2 \exp\left(-\frac{c_0\big(\delta_1 - \rho_k \sqrt{\ns \cE_\As}\big)^2}{\sum_{\ell=1}^{\ns} b_\ell^2}\right),
\end{align}
where $(d)$ is mainly due to the bound on $\|M_0\|_2$ in \eqref{eqn:bound_M0_H0}, while $(e)$ follows from the Banach-space-valued Azuma inequality in Appendix~\ref{app:azuma}. In addition, we can establish using \eqref{eqn:martingale_bdd_3}, the inequality $\sum_{j > \ell} \|\theta_j\|_2 \leq \sqrt{\ns \cE_\As}$, and some tedious algebraic manipulations that
\begin{align}
\label{eqn:azuma_bds_1}
    \sum_{\ell=1}^{\ns} b_\ell^2 = \gamma^2_{2,k} \sum_{\ell=1}^{\ns} \left(\|\theta_\ell\|_2 +
        \frac{\sum_{j > \ell} \|\theta_j\|_2}{\Ns-\ell-1}\right)^2 \leq \gamma^2_{2,k} \cE_\As \left(\frac{\Ns}{\Ns-\ns}\right)^2.
\end{align}
Combining \eqref{eqn:prob_bd_H0}, \eqref{eqn:azuma_1} and \eqref{eqn:azuma_bds_1}, we therefore obtain $\Pr(Z_1^k \geq \delta_1 \big| \cH_0^k) \leq e^2 \exp\left(-\frac{c_0(\Ns-\ns)^2\big(\delta_1 - \rho_k \sqrt{\ns \cE_\As}\big)^2}{\Ns^2 \gamma^2_{2,k} \cE_\As}\right)$.

We now complete the proof by noting that
\begin{align}
\nonumber
    \Pr\left(T_k(y) \geq \tau \big| \cH^k_0\right) &= \Pr\left(\wtT_k(y) \geq \sqrt{\tau} \big| \cH^k_0\right) \leq \Pr\left(Z_1^k + Z_2^k \geq \sqrt{\tau} \big| \cH^k_0\right)\\
\nonumber
        &\leq \Pr\left(Z_1^k + Z_2^k \geq \sqrt{\tau} \big| \cH^k_0, Z_2^k < \epsilon_2\right) + \Pr\left(Z_2^k \geq \epsilon_2 \big| \cH^k_0\right)\\
        &\leq \Pr\left(Z_1^k \geq \sqrt{\tau} - \epsilon_2 \big| \cH^k_0\right) + \Pr\left(Z_2^k \geq \epsilon_2\right).
\end{align}
The two statements in the lemma now follow from the (probabilistic) bounds on $Z_2^k$ established at the start of the proof and the probabilistic bound on $Z_1^k$ obtained in the preceding paragraph.\qed

\section{\revise{Proof of Lemma~\ref{lemma:left_tail_alt_hypo}}}\label{app:lemma2}
We once again define $\wtT_k(y) := \sqrt{T_k(y)}$ and note that $\wtT_k(y) \geq \big\|\sum_{j=1}^\ns \Phi_k^\tT\Phi_{i_j} \theta_j\big\|_2 - \big\|\Phi_k^\tT\eta\big\|_2$. Therefore, characterization of the left-tail probability of $Z_1^k := \big\|\sum_{j=1}^\ns \Phi_k^\tT\Phi_{i_j} \theta_j\big\|_2$ and the right-tail probability of $Z_2^k := \big\|\Phi_k^\tT\eta\big\|_2$ under $\cH_1^k$ helps us specify the left-tail probability of $T_k(y)$ under $\cH_1^k$. Since the right-tail probability of $Z_2^k$ for both deterministic and stochastic errors has already been specified in the proof of Lemma~\ref{lemma:right_tail_null_hypo}, we need only focus on the left-tail probability of $Z_1^k$ under $\cH_1^k$ in here.

In order to characterize $\Pr(Z_1^k \leq \delta_1 \big| \cH_1^k)$, we once again define $\bar{\Pi} := \left(\pi_1,\dots,\pi_\Ns\right)$ to be a random permutation of $\{1,\dots,\Ns\}$ and use $\Pi := \left(\pi_1,\dots,\pi_\ns\right)$ to denote the first $\ns$-elements of $\bar{\Pi}$. We then have the following equality in distribution:
\begin{align}
\label{eqn:left_tail_eq_in_distrib_1}
    \Big\|\sum_{j=1}^\ns \Phi_k^\tT\Phi_{i_j} \theta_j \Big\|_2 \ : \ k \in \As \ \overset{dist}{=} \
    \Big\|\sum_{j=1}^\ns \Phi_k^\tT\Phi_{\pi_j} \theta_j\Big\|_2 \ : \ k \in \Pi.
\end{align}
We now define a probability event $E^k_1 := \big\{\Pi = \left(\pi_1,\dots,\pi_\ns\right) : k \in \Pi\big\}$ and notice from \eqref{eqn:left_tail_eq_in_distrib_1} that
\begin{align}
\label{eqn:prob_bd_H1}
    \Pr(Z_1^k \leq \delta_1 \big| \cH_1^k) =
        \Pr\Bigg(\Big\|\sum_{j=1}^\ns \Phi_k^\tT\Phi_{\pi_j} \theta_j\Big\|_2 \leq
        \delta_1 \big| E^k_1\Bigg).
\end{align}
Next, we fix an arbitrary $i \in \{1,\dots,\ns\}$ and define another probability event $E_2^i := \{\pi_i = k\}$. It then follows that
\begin{align}
\nonumber
    \Pr\Bigg(\Big\|\sum_{j=1}^\ns \Phi_k^\tT\Phi_{\pi_j} \theta_j\Big\|_2 \leq
            \delta_1 \big| E^k_1\Bigg) &= \sum_{i=1}^{\ns} \Pr\Bigg(\Big\|\sum_{j=1}^\ns \Phi_k^\tT\Phi_{\pi_j} \theta_j\Big\|_2 \leq
            \delta_1 \big| E^k_1, E_2^i\Bigg) \Pr(E_2^i \big| E^k_1)\\
\nonumber
            &= \sum_{i=1}^{\ns} \Pr\Bigg(\Big\|\theta_i + \sum_{\substack{j=1\\j\not=i}}^\ns \Phi_k^\tT\Phi_{\pi_j} \theta_j\Big\|_2 \leq
                \delta_1 \big| E^k_1, E_2^i\Bigg) \Pr(E_2^i \big| E^k_1)\\
\label{eqn:prob_bd_H1_2}
            &\stackrel{(a)}{\leq} \sum_{i=1}^{\ns} \Pr\Bigg(\Big\|\sum_{\substack{j=1\\j\not=i}}^\ns \Phi_k^\tT\Phi_{\pi_j} \theta_j\Big\|_2 \geq
                \sqrt{\cE_k} - \delta_1 \big| E_2^i\Bigg) \Pr(E_2^i \big| E^k_1),
\end{align}
where $(a)$ follows for the facts that ($i$) $\|\theta_i + \sum_{j\not=i} \Phi_k^\tT\Phi_{\pi_j} \theta_j\|_2 \geq \|\theta_i\|_2 - \|\sum_{j\not=i} \Phi_k^\tT\Phi_{\pi_j} \theta_j\|_2$, ($ii$) $\|\theta_i\|_2$ conditioned on $E_2^i$ is $\sqrt{\cE_k}$, and ($iii$) $E_2^i \subset E^k_1$. It can be seen from \eqref{eqn:prob_bd_H1} and \eqref{eqn:prob_bd_H1_2} that our main challenge now becomes specifying the right-tail probability of $\|\sum_{j\not=i} \Phi_k^\tT\Phi_{\pi_j} \theta_j\|_2$ conditioned on $E_2^i$. To this end, we once again rely on Proposition~\ref{prop:azumaineq} in Appendix~\ref{app:azuma}.

Specifically, we construct an $\R^\ds$-valued Doob martingale $(M_0,M_1,\dots,M_{\ns-1})$ on $\sum_{j\not=i} \Phi_k^\tT\Phi_{\pi_j} \theta_j$ as follows. We first define $\Pi^{-i} := \left(\pi_1,\dots,\pi_{i-1},\pi_{i+1},\dots,\pi_\ns\right)$ and then define
\begin{align}
    M_0 &:= \sum_{\substack{j=1\\j\not=i}}^\ns \Phi_k^\tT \E\big[\Phi_{\pi_j}\big|E_2^i\big] \theta_j, \quad \text{and}\\
    M_{\ell} &:= \sum_{\substack{j=1\\j\not=i}}^\ns \Phi_k^\tT \E\big[\Phi_{\pi_j}\big|\pi_1^{-i,\ell}, E_2^i\big] \theta_j, \ \ell=1,\dots,\ns-1,
\end{align}
where $\pi_1^{-i,\ell}$ denotes the first $\ell$ elements of $\Pi^{-i}$. The next step in the proof involves showing $\|M_\ell - M_{\ell-1}\|_2$ is bounded for all $\ell \in \{1,\dots,\ns-1\}$. To do this, we use $\pi_\ell^{-i}$ to denote the $\ell$-th element of $\Pi^{-i}$ and define
\begin{align}
    M_{\ell}(u) := \sum_{\substack{j=1\\j\not=i}}^\ns \Phi_k^\tT \E\big[\Phi_{\pi_j}\big|\pi_1^{-i,\ell-1}, \pi_\ell^{-i} = u, E_2^i\big] \theta_j
\end{align}
for $u \in \{1,\dots,\Ns\} \setminus \{k\}$ and $\ell=1,\dots,\ns-1$. It then follows from the argument in Lemma~\ref{lemma:right_tail_null_hypo} that $\|M_\ell - M_{\ell-1}\|_2 \leq \sup_{u,v} \|M_\ell(u) - M_\ell(v)\|_2$. We now define a $\Dh \times \ds$ matrix $\tPhi_{\ell,j}^{u,v}$ as
\begin{align}
    \tPhi_{\ell,j}^{u,v} := \E\big[\Phi_{\pi_j}\big|\pi_1^{-i,\ell-1}, \pi_\ell^{-i} = u, E_2^i\big] - \E\big[\Phi_{\pi_j}\big|\pi_1^{-i,\ell-1}, \pi_\ell^{-i} = v, E_2^i\big], \quad \ell=1,\dots,\ns.
\end{align}
It is then easy to see that $\forall j > \ell + 1, j \not= i$, the random variable $\pi_j$ conditioned on the events $\{\pi_1^{-i,\ell-1}, \pi_\ell^{-i} = u, E_2^i\}$ and $\{\pi_1^{-i,\ell-1}, \pi_\ell^{-i} = v, E_2^i\}$ has a uniform distribution over the sets $\{1,\dots,\Ns\} \setminus \{\pi_1^{-i,\ell-1}, u, k\}$ and $\{1,\dots,\Ns\} \setminus \{\pi_1^{-i,\ell-1}, v, k\}$, respectively. It therefore follows $\forall j > \ell + 1, j \not= i$ that $\tPhi_{\ell,j}^{u,v} = \frac{1}{\Ns - \ell - 1}(\Phi_u - \Phi_v)$.

In order to evaluate $\tPhi_{\ell,j}^{u,v}$ for $j \leq \ell+1, j \not= i$, we need to consider three cases for the index $\ell$. In the first case of $\ell \geq i$, it can be seen that $\tPhi_{\ell,j}^{u,v} = 0~\forall j \leq \ell$ and $\tPhi_{\ell,j}^{u,v} = \Phi_u - \Phi_v$ for $j = \ell + 1$. In the second case of $\ell = i - 1$, it can similarly be seen that $\tPhi_{\ell,j}^{u,v} = 0~\forall j < \ell$ and $j = \ell + 1$, while $\tPhi_{\ell,j}^{u,v} = \Phi_u - \Phi_v$ for $j = \ell$. In the final case of $\ell < i - 1$, it can be further argued that $\tPhi_{\ell,j}^{u,v} = 0~\forall j < \ell$, $\tPhi_{\ell,j}^{u,v} = \Phi_u - \Phi_v$ for $j = \ell$, and $\tPhi_{\ell,j}^{u,v} = \frac{1}{\Ns - \ell - 1}(\Phi_u - \Phi_v)$ for $j = \ell + 1$. Combining all these facts together, we have the following upper bound:
\begin{align}
\nonumber
    \|M_\ell(u) - M_\ell(v)\|_2 &= \big\|\sum_{\substack{j=1\\j\not=i}}^\ns \Phi_k^\tT \tPhi_{\ell,j}^{u,v} \theta_j\big\|_2 \stackrel{(b)}{\leq}
            \sum_{\substack{j\geq\ell\\j\not=i}} \big\|\Phi_k^\tT \tPhi_{\ell,j}^{u,v}\big\|_2 \|\theta_j\|_2\\
\nonumber
        &\stackrel{(c)}{\leq} \big\|\Phi_k^\tT\left(\Phi_u - \Phi_v\right)\big\|_2 \Bigg(\|\theta_\ell\|_2 1_{\{\ell \not= i\}} + \|\theta_{\ell+1}\|_2 1_{\{\ell \not= i-1\}} +
            \sum_{\substack{j>\ell+1\\j\not=i}} \frac{\|\theta_j\|_2}{\Ns - \ell - 1}\Bigg)\\
\label{eqn:martingale_bdd_H1_1}
            &\leq \big(\gamma(\cS_k, \cS_u) + \gamma(\cS_k, \cS_v)\big) \Bigg(\|\theta_\ell\|_2 1_{\{\ell \not= i\}} + \|\theta_{\ell+1}\|_2 1_{\{\ell \not= i-1\}} +
            \sum_{\substack{j>\ell+1\\j\not=i}} \frac{\|\theta_j\|_2}{\Ns - \ell - 1}\Bigg).
\end{align}
Here, $(b)$ and $(c)$ follow from the preceding facts that $\tPhi_{\ell,j}^{u,v} = 0~\forall j < \ell$ and $\big\|\Phi_k^\tT \tPhi_{\ell,j}^{u,v}\big\|_2 \leq \big\|\Phi_k^\tT\left(\Phi_u - \Phi_v\right)\big\|_2$ for $j = \ell$ and $j = \ell + 1$. Consequently, it follows from \eqref{eqn:martingale_bdd_H1_1} and definition of the local $2$-subspace coherence that
\begin{align}
\label{eqn:martingale_bdd_H1}
    \|M_\ell - M_{\ell-1}\|_2 \leq
        \underbrace{\gamma_{2,k}\Bigg(\|\theta_\ell\|_2 1_{\{\ell \not= i\}} + \|\theta_{\ell+1}\|_2 1_{\{\ell \not= i-1\}} +
            \sum_{\substack{j>\ell+1\\j\not=i}} \frac{\|\theta_j\|_2}{\Ns - \ell - 1}\Bigg)}_{b_\ell}.
\end{align}
The next step needed to utilize Proposition~\ref{prop:azumaineq} involves an upper bound on $\|M_0\|_2$, which is given as follows:
\begin{align}
\label{eqn:bound_M0_H1}
    \|M_0\|_2 = \Big\|\sum_{j \not= i} \Phi_k^\tT \big(\sum_{\substack{q=1\\q\not=k}}^{\Ns}\frac{\Phi_q}{\Ns - 1}\big) \theta_j\Big\|_2 \leq
                        \frac{1}{\Ns-1}\Big\|\sum_{\substack{q=1\\q\not=k}}^{\Ns} \Phi_k^\tT \Phi_q\Big\|_2 \Big\|\sum_{j\not= i}\theta_j\Big\|_2
                        \stackrel{(d)}{\leq} \rho_k \sqrt{(\ns-1)(\cE_\As - \cE_k)}.
\end{align}
Here, $(d)$ primarily follows from the fact that, conditioned on $E_2^i$, $\sum_{j\not= i} \|\theta_j\|_2^2 = \cE_\As - \cE_k$.

Our construction of the Doob martingale, Proposition~\ref{prop:azumaineq} in Appendix~\ref{app:azuma}, \cite[Lemma~B.1]{Donahue.etal.CA1997} and the assumption $\sqrt{\cE_k} - \delta_1 > \rho_k \sqrt{\ns(\cE_\As - \cE_k)}$ now provides us the following upper bound:
\begin{align}
\nonumber
    \Pr\Bigg(\Big\|\sum_{\substack{j=1\\j\not=i}}^\ns \Phi_k^\tT\Phi_{\pi_j} \theta_j\Big\|_2 \geq
                \sqrt{\cE_k} - \delta_1 \big| E_2^i\Bigg) &= \Pr\big(\|M_{\ns-1}\|_2 \geq
        \sqrt{\cE_k} - \delta_1 \big| E_2^i\big)\\
\nonumber
        &= \Pr\big(\|M_{\ns-1}\|_2 - \|M_0\|_2 \geq \sqrt{\cE_k} - \delta_1 - \|M_0\|_2\big| E_2^i\big)\\
\nonumber
        &\stackrel{(e)}{\leq} \Pr\left(\|M_{\ns-1} - M_0\|_2 \geq
        \sqrt{\cE_k} - \delta_1 - \rho_k \sqrt{\ns(\cE_\As - \cE_k)} \,\big| E^k_0\right)\\
\label{eqn:azuma_2}
    &\leq e^2 \exp\left(-\frac{c_0\big(\sqrt{\cE_k} - \delta_1 - \rho_k \sqrt{\ns(\cE_\As - \cE_k)}\big)^2}{\sum_{\ell=1}^{\ns-1} b_\ell^2}\right),
\end{align}
where $(e)$ is primarily due to the bound on $\|M_0\|_2$ in \eqref{eqn:bound_M0_H1}. Further, it can be shown using \eqref{eqn:martingale_bdd_H1}, the inequality $\sum_{\ell=1}^{\ns-1}\|\theta_\ell\|_2 1_{\{\ell \not= i\}}\cdot\|\theta_{\ell+1}\|_2 1_{\{\ell \not= i-1\}} \leq (\cE_\As - \cE_k)$, and some tedious manipulations that the following holds:
\begin{align}
\label{eqn:azuma_bds_H1}
    \sum_{\ell=1}^{\ns-1} b_\ell^2 \leq \gamma^2_{2,k}(\cE_\As - \cE_k)\left(\frac{2\Ns - \ns}{\Ns-\ns}\right)^2.
\end{align}
Combining \eqref{eqn:prob_bd_H1}, \eqref{eqn:prob_bd_H1_2}, \eqref{eqn:azuma_2} and \eqref{eqn:azuma_bds_H1}, we obtain $\Pr(Z_1^k \leq \delta_1 \big| \cH_1^k) \leq e^2 \exp\left(-\frac{c_0(\Ns-\ns)^2\big(\sqrt{\cE_k} - \delta_1 - \rho_k \sqrt{\ns(\cE_\As - \cE_k)}\big)^2}{(2\Ns - \ns)^2 \gamma^2_{2,k} (\cE_\As - \cE_k)}\right)$.

The proof of the lemma can now be completed by noting that
\begin{align}
\nonumber
    \Pr\left(T_k(y) \leq \tau \big| \cH^k_1\right) &= \Pr\left(\wtT_k(y) \leq \sqrt{\tau} \big| \cH^k_1\right) \leq \Pr\left(Z_1^k - Z_2^k \leq \sqrt{\tau} \big| \cH^k_1\right)\\
\nonumber
        &\leq \Pr\left(Z_1^k - Z_2^k \leq \sqrt{\tau} \big| \cH^k_1, Z_2^k < \epsilon_2\right) + \Pr\left(Z_2^k \geq \epsilon_2 \big| \cH^k_1\right)\\
        &\leq \Pr\left(Z_1^k \leq \sqrt{\tau} + \epsilon_2 \big| \cH^k_1\right) + \Pr\left(Z_2^k \geq \epsilon_2\right).
\end{align}
The two statements in the lemma now follow from the (probabilistic) bounds on $Z_2^k$ established at the start of the proof of Lemma~\ref{lemma:right_tail_null_hypo} and the probabilistic bound on $Z_1^k$ obtained in the preceding paragraph.\qed

\section{\revise{Proof of Lemma~\ref{lemma:mix_coh_ubound_RDmodel}}}\label{app:lemma3}
\revise{We begin with \emph{any} arbitrary (but fixed) collection of orthonormal bases of the subspaces $\{\cS_i\}_{i=1}^{\Ns}$, denoted by $\big\{\Psi_i \in \R^{\Dh \times \ds}\big\}_{i=1}^{\Ns}$. Next, let $\big\{R_i \in \R^{\ds \times \ds}\big\}_{i=1}^{\Ns}$ be a collection of random rotation matrices that are drawn in an independent manner using the Haar measure, $\lambda_R$, on the space $O(d)$ of $d \times d$ rotation matrices. Given these $R_i$'s, notice that $\{R_i \Psi_i\}_{i=1}^{\Ns}$ also form a collection of orthonormal bases of the subspaces $\{\cS_i\}_{i=1}^{\Ns}$. Our goal now is to leverage the probabilistic method and establish that
\begin{align}\label{eqn:proof.lemma3.prob.arg}
  \Pr\Bigg(\bigcap_{i=1}^{\Ns} \Bigg\{\frac{1}{\Ns-1}\Big\|\sum_{j \not= i} R_i^\tT \Psi_i^\tT \Psi_j R_j\Big\|_2 < \bar{\rho}_i \Bigg\}\Bigg) > 0.
\end{align}
Assuming \eqref{eqn:proof.lemma3.prob.arg} holds, it then follows that there exists a \emph{deterministic} collection $\cQ_\Ns = \big\{Q_i \in O(d)\big\}_{i=1}^{\Ns}$ such that
\begin{align}\label{eqn:proof.lemma3.set.existence}
  \frac{1}{\Ns-1}\Big\|\sum_{j\not=i} (Q_i \Psi_i)^\tT (Q_j \Psi_j)\Big\|_2 < \bar{\rho}_{i}, \quad i=1,\dots,\Ns.
\end{align}
We can afterward define the promised bases as $U_i :=  Q_i \Psi_i$, which then completes the proof of the lemma.

In order to establish \eqref{eqn:proof.lemma3.prob.arg}, notice that
\begin{align}
    \Pr\Bigg(\bigcap_{i=1}^{\Ns} \Bigg\{\frac{1}{\Ns-1}\Big\|\sum_{j \not= i} R_i^\tT \Psi_i^\tT \Psi_j R_j\Big\|_2 < \bar{\rho}_i \Bigg\}\Bigg) &= 1 - \Pr\Bigg(\bigcup_{i=1}^{\Ns} \Bigg\{\frac{1}{\Ns-1}\Big\|\sum_{j \not= i} R_i^\tT \Psi_i^\tT \Psi_j R_j\Big\|_2 \geq \bar{\rho}_i \Bigg\}\Bigg)\nonumber\\
    &\geq 1 - \sum_{i=1}^{\Ns} \Pr\Bigg(\frac{1}{\Ns-1}\Big\|\sum_{j \not= i} R_i^\tT \Psi_i^\tT \Psi_j R_j\Big\|_2 \geq \bar{\rho}_i\Bigg).\label{eqn:proof.lemma3.prob.arg.2}
\end{align}
Thus, if we can establish that each term in the summation in \eqref{eqn:proof.lemma3.prob.arg.2} is strictly upper bounded by $N^{-1}$ then that equivalently establishes \eqref{eqn:proof.lemma3.prob.arg}. To this end, we first fix the index $i=1$ since identical results for other indices follow in a similar manner. Next, let $\|B\|_{S(p)}, 1 \leq p < \infty$, denote the Schatten $p$-norm of the matrix $B$, defined as $\|B\|_{S(p)} := \left(\sum_{k \geq 1} s_k^p(B)\right)^{1/p}$, where $s_k(B)$ denotes the $k$-th largest singular value of $B$~\cite{Horn.Johnson.Book1994}. It then follows from the definitions of $\|\cdot\|_2$ and $\|\cdot\|_{S(p)}$ that
\begin{align}\label{eqn:proof.lem3.schatten.norm.inequality}
  \frac{1}{\Ns-1}\Big\|\sum_{j > 1} R_1^\tT \Psi_1^\tT \Psi_j R_j\Big\|_2 \leq \rho_{1,S(p)} := \frac{1}{\Ns-1}\Big\|\sum_{j>1} R_1^\tT \Psi_1^\tT \Psi_j R_j\Big\|_{S(p)} \leq \frac{d^{1/p}}{\Ns-1}\Big\|\sum_{j > 1} R_1^\tT \Psi_1^\tT \Psi_j R_j\Big\|_2.
\end{align}
We therefore have from \eqref{eqn:proof.lem3.schatten.norm.inequality} that $\Pr\big(\frac{1}{\Ns-1}\|\sum_{j > 1} R_1^\tT \Psi_1^\tT \Psi_j R_j\|_2 \geq \bar{\rho}_1\big) \leq \Pr(\rho_{1,S(p)} \geq \bar{\rho}_1)$.

In order to bound $\Pr(\rho_{1,S(p)} \geq \bar{\rho}_1)$, we once again utilize Proposition~\ref{prop:azumaineq} in Appendix~\ref{app:azuma}. To this end, we construct an $\R^{\ds \times \ds}$ matrix-valued Doob's martingale $(M_0,M_1,\dots,M_{\Ns-1})$ as follows: $M_0 \equiv 0$ and
\begin{align}\label{eqn:proof.lemma3.Doob.martingale}
  M_\ell = \sum_{j=2}^{\Ns} R_1^\tT \Psi_1^\tT \Psi_j \E\big[R_j|R_1,R_2,\dots,R_{\ell+1}\big] \stackrel{(a)}{\equiv} \sum_{j=2}^{\ell+1} R_1^\tT \Psi_1^\tT \Psi_j R_j, \quad \ell=1,\dots,\Ns-1,
\end{align}
where $(a)$ follows from the mutual independence and zero mean of the random rotation matrices. Next, notice that
\begin{align}\label{eqn:proof.lemma3.martingale.bound}
    \forall \ell \geq 1, \quad  \|M_\ell - M_{\ell-1}\|_{S(p)} &= \|R_1^\tT \Psi_1^\tT \Psi_{\ell+1} R_{\ell+1}\|_{S(p)} \leq d^{1/p} \|R_1^\tT \Psi_1^\tT \Psi_{\ell+1} R_{\ell+1}\|_2 \nonumber \\
        &\leq d^{1/p} \|R_1\|_2 \|\Psi_1^\tT \Psi_{\ell+1}\|_2 \|R_{\ell+1}\|_2 = d^{1/p} \gamma(\cS_1, \cS_{\ell+1}),
\end{align}
Finally, in order to translate Proposition~\ref{prop:azumaineq} for $(\cB, \|\cdot\|) \equiv (\R^{d \times d}, \|\cdot\|_{S(p)})$, note that $\forall p \geq 2, \zeta_{\cB}(\tau) \leq \tfrac{p-1}{2} \tau^2$ \cite{Naor.CPaC2012}. It then follows that
\begin{align}\label{eqn:proof.lemma3.azuma.ineq}
    \Pr(\rho_{1,S(p)} \geq \bar{\rho}_1) &= \int_{R_1 \in O(d)} \Pr\Big(\|M_{N-1}\|_{S(p)} \geq (N-1)\bar{\rho}_1\,\big|\,R_1\Big)\lambda_R(d R_1) \nonumber\\
    &\leq e^{\max\{\frac{p}{2},2\}} \exp\bigg(-\frac{c_0(N-1)^2\bar{\rho}_1^2}{d^{2/p} \sum_{j > 1} \gamma^2(\cS_1, \cS_j)}\bigg) \int_{R_1 \in O(d)} \lambda_R(d R_1) \nonumber \\
    &= e^{\max\{\frac{p}{2},2\}} \exp\bigg(-\frac{c_0(N-1) \bar{\rho}_1^2}{d^{2/p} \gamma^2_{\textsf{rms},1}}\bigg).
\end{align}
Finally, replacing $p = 4\log(d)$ and $\bar{\rho}_{1} = \frac{\gamma_{\emph{\textsf{rms}},1} \sqrt{\log(c_4 \ds^2\Ns)}}{\sqrt{c_0'(\Ns-1)}}$ in \eqref{eqn:proof.lemma3.azuma.ineq} results in $\Pr(\rho_{1,S(p)} \geq \bar{\rho}_1) \leq (c_4 N)^{-1} < N^{-1}$. This suffices to establish the statement of the lemma.}

\section{Banach-Space-Valued Azuma's Inequality}\label{app:azuma}
In this appendix, we state a Banach-space-valued concentration inequality from \cite{Naor.CPaC2012} that is central to some of the proofs in this paper.
\begin{proposition}[Banach-Space-Valued Azuma's Inequality]\label{prop:azumaineq} Fix $s > 0$ and assume that a Banach space $(\cB,
\|\cdot\|)$ satisfies
\begin{align*}
    \zeta_{\cB}(\tau) := \sup_{\substack{u,v\in\cB\\\|u\|=\|v\|=1}} \left\{\frac{\|u + \tau v\| + \|u - \tau v\|}{2} - 1\right\} \leq s\tau^2
\end{align*}
for all $\tau > 0$. Let $\{M_k\}_{k=0}^{\infty}$ be a $\cB$-valued martingale satisfying the pointwise bound  $\|M_k - M_{k-1}\| \leq b_k$ for all $k \in \N$, where $\{b_k\}_{k=1}^{\infty}$ is a sequence of positive numbers. Then for every $\delta > 0$ and $k \in \N$, we have
\begin{align*}
\Pr\left(\|M_k - M_0\| \geq \delta\right) \leq e^{\max\{s,2\}} \exp\bigg(-\frac{c_0\delta^2}{\sum_{\ell=1}^{k} b_\ell^2}\bigg),
\end{align*}
where $c_0 := \frac{e^{-1}}{256}$ is an absolute constant.
\end{proposition}
\begin{remark}
Theorem 1.5 in \cite{Naor.CPaC2012} does not explicitly specify $c_0$ and also states the constant in front of $\exp(\cdot)$ to be $e^{s+2}$. Proposition~\ref{prop:azumaineq} stated in its current form, however, can be obtained from the proof of Theorem 1.5 in \cite{Naor.CPaC2012}.
\end{remark}
\end{appendices}


\end{document}